\documentclass{article}
\pdfoutput=1

\usepackage{amsthm,amsfonts,amssymb,amsmath,amscd} % AMS packages
\usepackage{color}
\usepackage{mathrsfs}
\usepackage{comment} % block comments
\usepackage{ifthen} % If then blocks
\usepackage{url}
\usepackage{xargs}
\usepackage{tikz}
\usepackage{xypic}
\xyoption{pdf}

\usepackage{aliascnt}
\usepackage{hyperref}
\definecolor{tocolor}{rgb}{.1,.1,.1}
\definecolor{urlcolor}{rgb}{.2,.2,.6}
\definecolor{linkcolor}{rgb}{.1,.1,.5}
\definecolor{citecolor}{rgb}{.4,.2,.1}
\hypersetup{backref=true, colorlinks=true, urlcolor=urlcolor, linkcolor=linkcolor, citecolor=citecolor}

\newcommandx{\thdef}[2]{
	\newaliascnt{#1}{theorem}  
	\newtheorem{#1}[#1]{#2}
	\aliascntresetthe{#1}  
	\newtheorem*{#1*}{#2}
	\expandafter\newcommand\expandafter{\csname #1autorefname\endcsname}{#2}
}
\makeatletter
\newtheorem*{rep@theorem}{\rep@title}
\newcommand{\newreptheorem}[2]{%
\newenvironment{rep#1}[1]{%
 \def\rep@title{#2 \ref{##1}}%
 \begin{rep@theorem}}%
 {\end{rep@theorem}}}
\makeatother

\newtheorem{theorem}{Theorem}[section]
\newreptheorem{theorem}{Theorem}

\thdef{lemma}{Lemma}
\thdef{corollary}{Corollary}
\thdef{conjecture}{Conjecture}
\newreptheorem{conjecture}{Conjecture}
\thdef{proposition}{Proposition}
\theoremstyle{definition}
\thdef{definition}{Definition}

\theoremstyle{remark}
\thdef{remark}{Remark}

\theoremstyle{remark}
\thdef{example}{Example}

\newcommand{\defn}[1]{{\bf\emph{#1}}} % for definitions

\newcommand{\set}[2]{\left\{#1\,\middle |\ #2\right\}}
\newcommand{\rbrac}[1]{\left(#1\right)} % Round brackets
\newcommand{\sbrac}[1]{\left[#1\right]} % Square brackets
\newcommand{\cbrac}[1]{\left\{#1\right\}} % Curly brackets
\newcommand{\abrac}[1]{\left\langle#1\right\rangle} % Angle brackets
\newcommandx{\liebrac}[2][1=]{\sbrac{#2}_{\shf{#1}}} % Lie bracket
\newcommandx{\courbrac}[2][1=]{\sbrac{#2}_{\gtshf{#1}}} % Lie bracket
\newcommandx{\inprod}[2][1=]{\abrac{#2}_{\spc{#1}}} % Lie bracket

\newcommand{\mapdef}[5]{
	\begin{array}{ccccc}
	#1 &:& #2 &\to& #3 \\
		&&  #4 &\mapsto& #5
	\end{array}
}
\newcommandx{\fn}[2][2=]{#1\ifthenelse{\equal{#2}{}}{}{\!\rbrac{{#2}}}} % Functions with optional arguments
\newcommandx{\id}[2][2=]{\fn{{\rm id}_{#1}}[#2]} % The identity function \id{X} = id_X

\newcommand{\complex}{\mathbb{C}}
\newcommand{\integer}{\mathbb{Z}}
\renewcommand{\natural}{\mathbb{N}}

\newcommand{\ext}[2][\bullet]{\spc{\Lambda}^{#1}{#2}} % exterior products
\newcommand{\sym}[2][\bullet]{\spc{S}^{#1}{#2}} % symmetric products
\newcommandx{\End}[2][1=]{\fn{\spc{End}_{#1}}[#2]}
\newcommandx{\Aut}[2][1=]{\fn{\spc{Aut}_{#1}}[#2]}
\newcommandx{\Hom}[2][1=]{\fn{\spc{Hom}_{#1}}[#2]}
\newcommandx{\image}[1]{\fn{\spc{img}}[#1]}
\renewcommandx{\ker}[1]{\fn{\spc{ker}}[#1]}
\newcommandx{\rank}[1]{\fn{\mathrm{rank}}[#1]}
\newcommandx{\ann}[1]{\fn{\spc{ann}}[\spc{#1}]} % annihilator of a subspace

\newcommand{\insertshfarg}[1]{\ifthenelse{\equal{#1}{\spc{}}}{}{\!\rbrac{#1}}} % Adds brackets for function arguments if argument is provided
\newcommand{\spc}[1]{\mathsf{#1}} % Bold font for spaces
\newcommandx{\shf}[4][1=,2=,4=]{\mathcal{#3}^{#1}_{#2}\insertshfarg{\spc{#4}}} % Sheaves
\newcommandx{\strc}[2][2=]{\shf[][\spc{#1}]{O}[#2]} % Structure sheaves
\newcommandx{\smooth}[3][1=\infty,3=]{\shf[#1][\spc{#2}]{C}[#3]}
\newcommandx{\tshf}[2][2=]{\shf[][\spc{#1}]{T}[#2]} % Tangent sheaves
\newcommandx{\cotshf}[2][2=]{\shf[*][\spc{#1}]{T}[#2]} % Cotangent sheaves
\newcommandx{\gtshf}[2][2=]{\shf[][\spc{\negthickspace #1}]{T\negmedspace\negmedspace\negmedspace T}[#2]} % Generalized tangent sheaf
\newcommandx{\conorm}[2][2=]{\shf[][\spc{#1}]{N}[#2]}
\newcommandx{\forms}[3][1=\bullet,3=]{\shf[#1][\spc{#2}]{\mathrm{\Omega}}[#3]} % Holomorphic differential forms
\newcommandx{\smoothforms}[3][1=\bullet,3=]{\shf[#1][\spc{#2},C^\infty]{\mathrm{\Omega}}[#3]} % Holomorphic differential forms
\newcommandx{\mvect}[2][1=\bullet]{ \ext[#1]{\tshf{#2}} }% Multivector fields
\newcommandx{\der}[3][1=\bullet,3=]{\shf[#1][#2]{\mathscr{X}}[#3]} % Multivector fields
\newcommandx{\aforms}[3][1=\bullet,3=]{\shf[#1][\shf{#2}]{\mathrm{\Omega}}[#3]} % Differential forms
\newcommandx{\Supp}[2][1=]{\fn{\spc{Supp}_{#1}}[\shf{#2}]} % Support of a sheaf
\newcommandx{\Ann}[2][1=]{\fn{\mathrm{Ann}_{#1}}[\shf{#2}]} % Annihilator of a sheaf

\newcommandx{\tb}[2][1=]{\spc{T}_{\!#1}\spc{#2}} % Tangent bundle
\newcommandx{\ctb}[2][1=]{\spc{T}_{\!#1}^*\spc{#2}} % Cotangent bundle
\newcommandx{\htb}[2][1=]{\spc{T}_{#1}^{1,0}\spc{#2}} % Holomorphic tangent bundle
\newcommandx{\atb}[2][1=]{\spc{T}_{#1}^{0,1}\spc{#2}} % Antiholomorphic tangent bundle
\newcommandx{\Jet}[3][1=]{\spc{J}^{#2}_{#1}#3}
\newcommandx{\sJet}[3][1=,2=]{\shf[#1][#2]{J}\!#3} % Jet sheaf

\renewcommandx{\d}[2][1=,2=]{\fn{d_{\shf{\!#1}}}[#2]} % Exterior differential
\newcommandx{\hook}[2][2=]{\fn{i_{#1}}[#2]} % Interior product
\newcommandx{\del}[2][1=,2=]{\fn{\partial_{#1}}[#2]} % Del
\newcommandx{\delbar}[2][1=,2=]{\fn{\overline\partial_{#1}}[#2]} % Delbar
\newcommandx{\cvrnt}[3][1=,3=]{\nabla_{\negmedspace#1}^{#3}#2} % Covariant differential
\newcommandx{\lie}[2][2=]{\fn{\mathscr{L}_{#1}}[#2]} % Lie derivative
\newcommandx{\alie}[3][3=]{\fn{\mathscr{L}_{#2}^{\shf{#1}}}[#3]} % Lie derivative
\newcommandx{\curv}[2][2=]{\fn{\fn{R}[#1]}[#2]} % Curvature of a connection

\newcommand{\insertrel}[1]{\ifthenelse{\equal{#1}{}}{}{,#1}}
\newcommand{\insertbullet}[1]{\ifthenelse{\equal{#1}{.}}{\bullet}{#1}}
\newcommandx{\hlgy}[3][1=\bullet,3=]{\spc{H}_{#1}^{#3}\!\rbrac{{#2}}} % Homology with coefficients
\newcommandx{\hlgymap}[4][1=\bullet,3=,4=]{\spc{H}_{#1}^{#4}\fn{#2}[#3]} % Induced map on homology
\newcommandx{\cohlgy}[3][1=\bullet,3=]{\spc{H}^{#1}_{#3}\!\rbrac{{#2}}} % Cohomology with coefficients
\newcommandx{\chow}[3][1=\bullet,3=]{\spc{A}^{#1}_{#3}\!\rbrac{{#2}}} % Cohomology with coefficients
\newcommandx{\cohlgymap}[4][1=\bullet,3=,4=]{\spc{H}^{#1}_{#4}\fn{#2}[#3]} % Induced map on cohomology
\newcommandx{\Ext}[3][1=\bullet,3=]{\fn{\spc{Ext}^{#1}_{#3}}[{#2}]}
\newcommandx{\sChern}[2][1=\bullet]{\fn{\spc{Chern}^{#1}}[\shf{#2}]} % Chern algebra of a sheaf
\newcommandx{\chern}[2][1=]{\fn{c_{#1}}[#2]} % Chern algebra of a sheaf
\newcommandx{\ch}[2][1=]{\fn{\mathrm{ch}_{#1}}[{#2}]} % Chern character
\newcommandx{\Chern}[2][1=\bullet]{\fn{\spc{Chern}^{#1}}[\spc{#2}]} % Chern algebra of a vector bundle
\newcommandx{\sPont}[2][1=\bullet]{\fn{\spc{Pont}^{#1}}[\shf{#2}]} % Pontrjagin algebra of a sheaf
\newcommandx{\Pont}[2][1=\bullet]{\fn{\spc{Pont}^{#1}}[\spc{#2}]} % Pontrjagin algebra of a vector bundle

\newcommandx{\rk}[2][1=]{\fn{\spc{rk}_{\le#1}}[#2]}

\hypersetup{pdftitle={Poisson modules and degeneracy loci},pdfauthor={Marco Gualtieri, Brent Pym}}
\newcommand{\X}{\spc{X}}
\newcommand{\Y}{\spc{Y}}

\newcommand{\U}{\spc{U}}
\newcommand{\Z}{\spc{Z}}

\newcommand{\Prj}[1]{\mathbb{P}^{{#1}}}

\newcommand{\Aff}[1]{\mathbb{A}^{{#1}}}
\newcommand{\Spec}{\spc{Spec}}

\newcommandx{\Sec}[2][1=]{\spc{Sec}_{#1}\!\rbrac{#2}}
\newcommandx{\Div}[2][1=]{\spc{Div}_{#1}\!\rbrac{#2}}
\newcommand{\acX}{\can_\X^{-1}}
\newcommand{\can}{\mathcal{\omega}}
\newcommand{\V}{\spc{V}} % Vanishing set
\newcommand{\W}{\spc{W}}
\newcommand{\g}{\mathfrak{g}} % Lie algebra
\newcommand{\sA}{\mathcal{A}} % anchored sheaf
\newcommand{\E}{\spc{E}} 
\newcommand{\sE}{\mathcal{E}} % coherent sheaf
\newcommand{\sL}{\mathcal{L}}

\newcommandx{\sF}{\mathcal{F}}
\newcommandx{\sFd}[2][2=]{\shf[\vee][#1]{F}[#2]}

\newcommandx{\N}{\spc{N}}
\newcommandx{\sN}[3][2=,3=]{\shf[#2][#1]{N}[#3]}

\newcommandx{\sK}[2][2=]{\shf[][#1]{K}[#2]}

\newcommandx{\sI}[1]{\mathcal{I}_{#1}}
\newcommandx{\sJ}[2][2=]{\shf[][#1]{J}[#2]}

\newcommand{\ps}{\sigma}
\newcommand{\pss}{\sigma^\sharp}

\newcommandx{\ips}[1][1=]{\iota_{\ps^{#1}}}
\newcommandx{\D}[3][1=,2=]{\spc{D}_{#1}^{#2}(#3)}
\newcommandx{\AD}[3][1=,2=]{\D[#1][#2]{\ifthenelse{\equal{#2}{}}{\ps_{#3}}{#3}}}

\newcommand{\res}[2]{\mathrm{Res}^{#1}(#2)}
\newcommand{\mres}[2]{\mathrm{Res}^{#1}_{mod}(#2)}
\newcommand{\Sing}[1]{\spc{Sing}(#1)}
\newcommandx{\ati}[2][1=]{\alpha_{#1}(#2)}
\newcommand{\bvb}[1]{\spc{\Lambda}^2\tb{#1}} % bivector bundle

\newcommand{\cvf}[1]{\partial_{#1}}

\newcommand{\sPois}[1]{\fn{\mathcal{P}ois}[#1]} % Poisson vector fields
\newcommand{\sHam}[1]{\fn{\mathcal{H}am}[#1]} % Hamiltonian vector fields
\newcommand{\sImg}[1]{\fn{\mathcal{I}mg}[#1]}
\newcommand{\sKer}[1]{\fn{\mathcal{K}er}[#1]}

\newcommandx{\sHom}[2][1=]{\fn{\mathcal{H}om_{#1}}[#2]}
\newcommand{\sEnd}[1]{\fn{\mathcal{E}nd}[#1]}
\newcommandx{\sExt}[3][1=\bullet,3=]{\fn{\mathcal{E}xt^{#1}_{#3}}[#2]}

\newcommand{\codim}[1]{\fn{ {\rm codim} } [#1] }

\newcommand{\tr}{\mathrm{Tr}}

\newcommand{\mainthm}{Let $(\X,\ps)$ be a connected Fano Poisson manifold of dimension $2n$.  Then either $\D[2n-2]{\sigma} = \X$ or $\D[2n-2]{\sigma}$ is a non-empty hypersurface in $\X$.  Moreover, $\D[2n-4]{\sigma}$ is non-empty and has at least one component of dimension $\ge 2n-3$.}
\newcommand{\singthm}{Suppose that $\X$ has even dimension $2n$ and that $\ps \in \cohlgy[0]{\X,\ext[2]{\tshf{\X}}}$ is a generically symplectic Poisson structure.  Let $\W$ be the singular locus of the degeneracy hypersurface $\D[2n-2]{\ps}$.  If $\W$ is nonempty, then every component of $\W$ has dimension $\ge 2n-3$.  Moreover, if $\dim \W = 2n-3$, then $\W$ is a Gorenstein scheme with dualizing sheaf $\can_\X^{-1}|_\W$, and its fundamental class is given by
$$
[\W] = c_1c_2-c_3,
$$
where $c_j$ is the $j^{th}$ Chern class of $\X$.}

\begin{document}
\title{\vspace{-3em}Poisson modules and degeneracy loci}
\author{Marco Gualtieri\footnote{University of Toronto;  \href{mailto:mgualt@math.toronto.edu}{mgualt@math.toronto.edu}} \and Brent Pym\footnote{University of Toronto;  \href{mailto:bpym@math.toronto.edu}{bpym@math.toronto.edu}}}
\maketitle

\begin{abstract}
In this paper, we study the interplay between modules and sub-objects in holomorphic Poisson geometry.  In particular, we define a new notion of ``residue'' for a Poisson module, analogous to the Poincar\'e residue of a meromorphic volume form.  Of particular interest is the interaction between the residues of the canonical line bundle of a Poisson manifold and its degeneracy loci---where the rank of the Poisson structure drops.  As an application, we provide new evidence in favour of Bondal's conjecture that the rank $\le 2k$ locus of a Fano Poisson manifold always has dimension $\ge 2k+1$.  In particular, we show that the conjecture holds for Fano fourfolds.  We also apply our techniques to a family of Poisson structures defined by Fe{\u\i}gin and Odesski{\u\i}, where the degeneracy loci are given by the secant varieties of elliptic normal curves.
\end{abstract}

\tableofcontents

\section{Introduction}
\label{sec:intro}
Let $\X$ be a complex manifold equipped with a holomorphic Poisson structure, which may be viewed as a bivector field $\ps$ with vanishing Schouten bracket %satisfying the condition that
\begin{align}
[\ps,\ps] = 0. \label{eqn:schouten}
\end{align}
%employing the Schouten bracket of multivector fields.
When $\ps$ is non-degenerate, its inverse defines a holomorphic symplectic form.

In this paper, we shall study the geometry of the \defn{degeneracy locus}
$$
\D[2k]{\ps} = \set{x \in \X}{\rank{ \pss_x : {\ctb[x]{X}} \to {\tb[x]{X}}} \le 2k },
$$
where the rank of the Poisson tensor, which is necessarily even, is bounded above by $2k$.  This locus may also be described as the zero set of the section
$$
\ps^{k+1} = \underbrace{\ps \wedge \cdots \wedge \ps}_{k+1{\rm\ times}} \in \cohlgy[0]{\X,\ext[2k+2]{\tshf{\X}}}.
$$

The degeneracy loci of vector bundle morphisms have been well studied in algebraic geometry.  For an overview of the subject, we refer the reader to Chapter 14 in \cite{Fulton1998} and the references therein (particularly \cite{Harris1984a} and \cite{Jozefiak1981}).  Of relevance to us here is the observation that, for a rank-$r$ holomorphic vector bundle $\E$ on $\X$, and a section $\rho \in \cohlgy[0]{\X,\ext[2]{\E}}$, the degeneracy locus $\D[2k]{\rho}$ has codimension at most ${r-2k \choose 2}$ in $\X$, provided it is nonempty.  When $\rho$ is generic in an appropriate sense, this upper bound is actually an equality.  Furthermore, if
$$
{r-2k \choose 2} \le \dim \X
$$
and if $\ext[2]{\E}$ is an ample vector bundle,
%\footnote{Recall that a vector bundle is ample if the hyperplane bundle $\strc{\proj{\E^*}}(1)$ is an ample line bundle on the projective bundle $\proj{\E^*}$.}
then $\D[2k]{\rho}$ is non-empty.

In light of this observation, it is perhaps somewhat surprising that Bondal has made the following
\begin{conjecture}[\cite{Bondal1993}]\label{con:bondal}
Let $(\X,\ps)$ be a connected Fano Poisson manifold\footnote{A \defn{Fano manifold} is a compact complex manifold for which the anti-canonical line bundle $\can_\X^{-1}=\det\tshf{\X}$ is ample.}.  Then for each integer $k \ge 0$ such that $2k < \dim \X$, the degeneracy locus $\D[2k]{\ps}$ is non-empty and has a component of dimension $\ge 2k+1$.
\end{conjecture}

%When we say that the dimension is $\ge 2k+1$, we mean that there is at least one irreducible component of dimension $\ge 2k+1$, although there may be other components of strictly lower dimension.

This conjecture implies, for example, that the zero locus of a Poisson tensor on any Fano manifold must contain a curve.  In dimension three, this conjecture is already a departure from the generic situation: $\bvb{X}$ has rank three, so a generic bivector field should only vanish at isolated points.  The effect is even more dramatic in dimension four, where $\bvb{X}$ has rank six, and so we expect that the generic section is non-vanishing.  Of course, Poisson structures are far from generic sections; they satisfy the nonlinear partial differential equation \eqref{eqn:schouten}, and it is this ``integrability'' which leads to larger degeneracy loci.  

In 1997, Polishchuk provided some evidence in favour of Bondal's conjecture:

\begin{theorem}[{\cite[Corollary 9.2]{Polishchuk1997}}]\label{thm:polish}
Let $(\X,\ps)$ be a connected Fano Poisson manifold.  If the Poisson structure generically has rank $2k$, then the degeneracy locus $\D[2k-2]{\ps}$ is non-empty and has a component of dimension $\ge 2k-1$.
\end{theorem}

In fact, Polishchuk gave the proof in the odd-dimensional ``non-degenerate'' case in which the dimension of $\X$ is $2k+1$, but as Beauville observes in \cite{Beauville2010}, his proof extends easily to the more general case stated here.  It follows immediately that Bondal's conjecture holds for Poisson structures whose rank is $\le 2$ everywhere, and, in particular, for Fano varieties of dimension $\le 3$.

The basic tool used in Polishchuk's proof is Bott's vanishing theorem for the characteristic classes of the normal bundle to a nonsingular foliation \cite{Bott1972}.  Polishchuk applies this theorem to the anti-canonical line bundle $\can_\X^{-1} = \det \tshf{X}$ of $\X$ in order to conclude that an appropriate power of its first Chern class vanishes on $\X \setminus \D[2k-2]{\ps}$.  This vanishing would contradict the ampleness of $\can_\X^{-1}$ unless $\D[2k-2]{\ps}$ had sufficiently large dimension.  The crucial observation that allows the application of Bott's theorem is that $\can_\X^{-1}$ is the determinant of the normal bundle to the $2k$-dimensional symplectic leaves of $\ps$.  One might like to prove the full conjecture now by repeating the argument on $\D[2k-2]{\ps}$.  However, $\omega_\X^{-1}$ can no longer be identified with the appropriate normal bundle, and so the theorem does not apply.  Moreover, $\D[2k-1]{\ps}$ is not Fano, so one cannot apply an inductive argument, either.

In this paper, we provide further evidence for the conjecture in the even-dimensional case:
\begin{reptheorem}{thm:bondal}
\mainthm
\end{reptheorem}
In particular, this result shows that Bondal's \autoref{con:bondal} holds for Fano varieties of dimension four.  The statement about $\D[2n-2]{\ps}$ is straightforward: this locus is the zero set of the section
$$
\ps^n\in \cohlgy[0]{\X,\det\tshf{\X}} = \cohlgy[0]{\X,\can_\X^{-1}},
$$
so it is an anticanonical hypersurface in $\X$ unless $\ps^n = 0$, in which case $\D[2n-2]{\ps}$ is all of $\X$.  If $\D[2n-2]{\ps}$ is an irreducible hypersurface, it is a Calabi-Yau variety which turns out to be highly singular.  We expect similar behaviour from the lower-rank degeneracy loci of $\ps$.

Showing that $\D[2n-4]{\ps}$ has a component of dimension $\ge 2n-3$ is the difficult part of the theorem.  The main idea of the proof is to exploit the fact that, while the restriction of $\can_\X^{-1}$ to the degeneracy locus $\D[2n-2]{\ps}$ is not the normal bundle to the symplectic leaves, it is a \defn{Poisson module}---that is, it carries a flat ``Poisson connection''.  While Bott's theorem does not apply directly to Poisson modules, we show that there is a nonempty Poisson subvariety of $\D[2n-2]{\ps}$ where the theorem does apply.  By bounding the dimension of this subvariety, we are able to bound the dimension of $\D[2n-4]{\ps}$, as in Polishchuk's proof. 

In order to obtain an appropriate analogue of Bott's theorem, we must develop several aspects of the geometry of Poisson modules which, to our knowledge, have yet to appear in the literature.  In particular, we associate to any Poisson line bundle a collection of closed Poisson subvarieties, indexed by even integers $2k$.  These subvarieties have the useful property that the Poisson module is flat along their symplectic leaves of dimension $2k$.  As a result, one may apply Bott's vanishing theorem to the Poisson module on these subvarieties.

We also construct a multivector field of degree $2k+1$ on the degeneracy locus $\D[2k]{\ps}$ which is canonically associated to a Poisson line bundle.  We call these tensors the \defn{residues} of the module because they are direct analogues of the Poincar\'e residue of a meromorphic form.  They are also closely related to the classes constructed by Brylinski and Zuckerman in \cite{Brylinski1999}.  Of course, the residue can only be nonzero if $\dim \D[2k]{\ps} \ge 2k+1$, and so the fact that these tensors are nonzero in many examples is, in a sense, a local explanation for the dimensions occurring in \autoref{con:bondal}.

The natural Poisson module structure on the canonical bundle $\can_\X$ is of particular importance.  We derive a simple formula for the residues of this module in terms of the tensors $\ps^{k+1}$ which define the degeneracy loci.  Using this formula, we show that the divisor on which a generically symplectic Poisson structure degenerates is highly singular:
\begin{reptheorem}{thm:adapt-sing}
\singthm
\end{reptheorem}

Bailey~\cite{Bailey2012} has recently shown that any generalized complex manifold is locally isomorphic to a product of a real symplectic manifold with a holomorphic Poisson manifold.  It follows that the type-change loci of a generalized complex manifold are locally described by the degeneracy loci of a holomorphic Poisson structure.  Our theorem therefore puts constraints on the structure of these type-change loci.

The paper is organized as follows.  We begin in \autoref{sec:basics} with a review of Poisson geometry.  In \autoref{sec:subvar}, we discuss Poisson subvarieties, paying special attention to the ``strong'' Poisson subvarieties.  In \autoref{sec:mod-geom}, we recall and develop the basic theory of Poisson modules.  Most notably, we define the residues of a Poisson line bundle.  \autoref{sec:resolve} introduces the degeneracy loci of a Poisson module, and in \autoref{sec:char-class}, we explain how these degeneracy loci may be used to obtain Bott-type vanishing theorems for Poisson modules.  In \autoref{sec:degen}, we apply the theory to the canonical module, obtaining formulae for its residues in terms of the Poisson tensor itself.  We end in \autoref{sec:elliptic} with a discussion of a family of Poisson structures introduced by Fe{\u\i}gin and Odesski{\u\i}~\cite{Feigin1989,Feigin1998} related to elliptic normal curves.  In these examples, the residues of the canonical module are all non-trivial and \autoref{thm:adapt-sing} has consequences for the secant varieties of elliptic normal curves of odd degree.

\begin{remark}
Due to the importance of singular spaces in this topic, we work in the category of schemes.  We invite the reader unfamiliar with schemes to think of them as manifolds; we hope that much of the paper will still be accessible.\qed
\end{remark}
\noindent
\emph{Notation}: Throughout this paper, the word ``scheme'' shall mean a separated scheme of finite type over the complex numbers $\complex$.  If $\X$ is a scheme, we denote by $\strc{\X}$ its structure sheaf.  If $\sE$ is a sheaf of abelian groups on $\X$, we denote by $\cohlgy{\X,\sE}$ the cohomology of $\X$ with values in $\sE$.  If $\sE$ is an $\strc{X}$-module we denote by $\sE^\vee = \sHom[\strc{X}]{\sE,\strc{X}}$ its dual.

We write $\forms[1]{X} = \forms[1]{{\X/\complex}}$ for the sheaf of differentials on $\X$ and $\tshf{X} =(\forms[1]{X})^\vee$ for the tangent sheaf of $\complex$-linear derivations $\strc{X} \to \strc{X}$.

For $k \in \natural$, we denote by $\forms[k]{X} = \ext[k]{\forms[1]{X}}$ the sheaf of $k$-forms, and by $\der[k]{\X} = (\forms[k]{X})^\vee$ the $\strc{X}$-module of alternating $k$-multilinear forms on $\forms[1]{X}$.  More generally, if $\sE$ is an $\strc{X}$-module, we set $\der[k]{\X}(\sE) = \sHom[\strc{X}]{{\forms[k]{X},\sE}}.$  Notice that, being the dual of a coherent $\strc{X}$-module, $\der[k]{\X}$ is always reflexive.

The direct sum $\der{\X} = \bigoplus_{k \ge 0} \der[k]{\X}$ comes equipped with the usual wedge product of alternating forms, making it into a graded-commutative algebra.  Similarly, $\der{\X}(\sE) = \bigoplus_{k\ge 0}\der[k]{\X}(\sE)$ is a graded $\der{\X}$-module.  When $\X$ is smooth, $\forms[1]{X}$ is locally-free and hence $\der{\X} \cong \mvect{X}$ is simply the exterior algebra of the tangent sheaf.  However, the situation is more delicate when $\X$ is singular.  In this case, we still have $\der[1]{\X} = \tshf{X}$ by definition, but the obvious map $\mvect{\X} \to \der{\X}$ need not be an isomorphism in higher degrees.  In order to avoid confusion with the existing literature on Poisson manifolds, we refer to sections of $\der[k]{\X}$ as \defn{$k$-derivations} and reserve the terminology ``multivector fields'' for sections of $\mvect{\X}$.   It is $\der{\X}$ that plays the more central role in this paper.

Notice that even when $\X$ is singular, the operations of Lie derivative, contraction with forms, Schouten bracket, etc., familiar from the the theory of multivector fields on smooth manifolds carry over to $\der{\X}$ simply because it is the dual of the differential graded algebra $(\forms{\X},d)$.
\\ \ \\
%\noindent
\emph{Acknowledgements}: We are grateful to Arnaud Beauville, Ragnar-Olaf Buchweitz, William Graham, Nigel Hitchin, David Morrison, Alexander Odesski\u\i, Alexander Polishchuk and Mike Roth for many helpful conversations.  This research was supported by an NSERC Canada Graduate Scholarship (Doctoral), an NSERC Discovery Grant and an Ontario ERA.
\pagebreak
\section{Review of Poisson geometry}
\label{sec:basics}
We begin with a brief review of Poisson geometry in the algebraic setting, and refer to Polishchuk's paper~\cite{Polishchuk1997} for a more thorough introduction.    One significant difference from the smooth setting is that the usual identification between Poisson brackets and bivector fields, which depends upon the morphism 
\begin{equation}\label{eq:bivto2der}
\mvect[2]{\X}\to \der[2]{\X}
\end{equation}
between bivectors and biderivations being an isomorphism, need no longer hold.
\begin{definition}
A \defn{Poisson scheme} is a pair $(\X,\ps)$, where $\X$ is a scheme and $\ps \in \cohlgy[0]{\X,\der[2]{\X}}$ is a $2$-derivation such that the $\complex$-bilinear morphism
$$
\mapdef{\cbrac{\cdot,\cdot}}{\strc{\X} \times \strc{\X}}{\strc{\X}}
								{(f,g)}{\ps(df\wedge dg)}
$$
defines a Lie algebra structure on $\strc{\X}$.  This Lie bracket is the \defn{Poisson bracket}.  We denote by $\pss : \forms[1]{X} \to \tshf{X}$ the $\strc{X}$-linear \defn{anchor map} defined by
$$
\pss(\alpha)(\beta) = \ps(\alpha\wedge\beta)
$$
for all $\alpha,\beta \in \forms[1]{X}$.  We say that $(\X,\ps)$ is a \defn{smooth} Poisson scheme if the underlying scheme $\X$ is smooth.

\begin{remark}
We reserve the term ``regular'' for a stronger notion than smoothness; see \autoref{defn:sing-pois}.\qed
\end{remark}

\end{definition}
% An example demonstrating the difference between 2-derivations and bivectors
\begin{example}\label{ex:cone}
Consider the quotient $\X=\Aff{2}/\integer_2$ of $\Aff{2} =\Spec(\complex[x,y])$ by the equivalence relation $(x,y)\sim (-x,-y)$.  The map $\complex[u,v,w] \to \complex[x,y]$ sending $u,v,w$ to $ x^2,xy,y^2$ gives an isomorphism from $\X$ to the quadric cone $uw=v^2$ in $\Aff{3}$, a surface with an ordinary double point singularity.  The Poisson structure $\cvf{x}\wedge\cvf{y}$ on $\Aff{2}$ descends to $\X$ as a $2$-derivation, but not as a bivector field, because the morphism~\eqref{eq:bivto2der}, an inclusion in this case, has image generated by $\{x\cvf{x}\wedge x\cvf{y},\ x\cvf{x}\wedge y\cvf{y},\ y\cvf{x}\wedge y\cvf{y}\}$.  Therefore, the morphism has a torsion cokernel, which is supported at the singular point and generated by the image of the Poisson tensor.\qed
\end{example}
%We define morphisms of Poisson schemes in the obvious way:
\begin{definition}
Let $(\X,\sigma)$ and $(\Y,\eta)$ be Poisson schemes, with corresponding brackets $\cbrac{\cdot,\cdot}_\X$ and $\cbrac{\cdot,\cdot}_\Y$.  A morphism $f : \X \to \Y$ is a \defn{Poisson morphism} if it preserves the Poisson brackets, ie.~the pull-back morphism $f^* : \strc{Y} \to f_*\strc{X}$ satisfies
$$
f^*\cbrac{g,h}_\Y = \cbrac{f^*g,f^*h}_\X
$$
for all $g,h\in\strc{Y}$.
\end{definition}
As in symplectic geometry, a Poisson structure has a distinguished subsheaf of its infinitesimal symmetries which is generated by functions, as follows.
%Let $(\X,\strc{X},\sigma)$ be a Poisson scheme with corresponding Poisson bracket $\cbrac{\cdot,\cdot}$.  In Poisson geometry, a key role is played by the infinitesimal symmetries of the Poisson structure.  We therefore recall:
\begin{definition}
A vector field $Z \in \tshf{X}$ is \defn{Hamiltonian} with respect to $\ps$ if there exists a function $f \in \strc{X}$ such that $Z = \pss(df)$.  We denote by 
$$
\sHam{\ps} = \sImg{{\pss \circ d : \strc{X} \to \tshf{X}}}
$$
the $\complex$-linear subsheaf generated by (locally-defined) Hamiltonian vector fields.

A vector field $Z \in \tshf{X}$ is \defn{Poisson} with respect to $\ps$ if it is an infinitesimal symmetry of $\ps$, ie. $\lie{Z}\ps = 0$.  We denote by
$$
\sPois{\ps} \subset \tshf{X}
$$
the $\complex$-linear subsheaf of (locally-defined) Poisson vector fields. 
\end{definition}
%Notice that a vector field $Z \in \tshf{X}$ is Poisson if and only if it preserves the Poisson bracket, ie.
%$$
%Z(\cbrac{f,g}) = \cbrac{Z(f),g} + \cbrac{f,Z(g)}
%$$
%for all $f,g \in \strc{X}$.  
%It therefore follows from the Jacobi identity that all Hamiltonian vector fields are Poisson:
It follows from the Jacobi identity that Hamiltonian vector fields are Poisson:
$$
\sHam{\ps} \subset \sPois{\ps}.
$$
The discrepancy between these sheaves will be of interest to us in~\autoref{sec:mod-geom}.
%the In general, these two sheaves are rather far from agreeing, as we shall see in our examples~\marco{Where is this?}.

Since $\ps^\sharp \circ d : (\strc{X},\cbrac{\cdot,\cdot}) \to (\tshf{X},[\cdot,\cdot])$ is a Lie algebra homomorphism, the $\strc{X}$-submodule $\sImg{\ps^\sharp} \subset \tshf{X}$ is closed under the Lie bracket of vector fields.  Therefore, it defines a (singular) foliation of the complex analytic space associated to $\X$---the well-known foliation by symplectic leaves.  However, it is important to note that these symplectic leaves may not define algebraic subschemes of $\X$; after all, they are solutions of a differential equation.

\section{Poisson subschemes}
\label{sec:subvar}
Let $(\X,\ps)$ be a Poisson scheme.  We say the Poisson scheme $(\Y,\eta)$ is a \defn{Poisson subscheme} of $(\X,\ps)$ if $\Y$ is a subscheme of $\X$ and the embedding $i : \Y \to \X$ is a Poisson morphism.  An open embedding clearly defines a Poisson subscheme in a unique way.
A closed subscheme of $\X$, however, need not admit a compatible Poisson structure.  When it does, this Poisson structure is unique:
%can be a Poisson subscheme in at most one way:
\begin{proposition}\label{prop:subscheme}
Let $(\X,\ps)$ be a Poisson scheme, and let $\Y$ be a closed subscheme with embedding $i : \Y \to \X$ and ideal sheaf $\sI{\Y}$.  Then the following are equivalent:
\begin{enumerate}
\item $\Y$ admits the structure of a Poisson subscheme
%\item $\pss(d\sI{\Y}) \subset \sI{\Y}\tshf{X}$
\item $\sI{\Y}$ is a sheaf of Poisson ideals, ie. $\cbrac{\sI{\Y},\strc{X}} \subset \sI{\Y}$
\item $Z(\sI{\Y}) \subset \sI{\Y}$ for all $Z \in \sHam{\ps}$
%\item $Z(\sI{\Y}) \subset \sI{\Y}$ for all $Z \in \sImg{\pss}$
\end{enumerate}
In this case, the induced Poisson structure on $\Y$ is unique and denoted $\ps|_\Y$.
\end{proposition}

\begin{proof}
The equivalence of 2) and 3) follows from the relation between the Poisson bracket and the morphism $\pss : \forms[1]{X} \to \tshf{X}$.  Meanwhile, condition 2) is exactly what is required for $\cbrac{\cdot,\cdot}$ to descend to a Poisson bracket on $\strc{Y} = \strc{X}/\sI{\Y}$.
\end{proof}

It will be useful for us to consider a special class of Poisson subschemes:
\begin{definition}
Let $(\X,\ps)$ be a Poisson scheme.  A closed subscheme $\Y$ of $\X$ is a \defn{strong Poisson subscheme} if its ideal sheaf $\sI{\Y}$ is preserved by all Poisson vector fields, ie. if
$$
Z(\sI{\Y}) \subset \sI{\Y}
$$
for all $Z \in \sPois{\ps}$.
\end{definition}

Note that the condition of being a strong Poisson subscheme is a local one.  In general, it is not sufficient to check that $\Y$ is preserved by all global Poisson vector fields $Z \in \cohlgy[0]{\X,\sPois{\ps}}$.  Clearly, every strong Poisson subscheme of $(\X,\ps)$ is a Poisson subscheme.  However, the converse need not hold in general, as the following example shows:
\begin{example}
Let $\Aff{3} = \Spec(\complex[x,y,z])$ be affine three-space.  The bivector field
$$
\ps = \cvf{x} \wedge \cvf{y}
$$
is Poisson, and the image of $\ps^\sharp$ is spanned by $\cvf{x}$ and $\cvf{y}$.  The $z=0$ plane $\W \subset \Aff{3}$ is therefore preserved by all Hamiltonian vector fields, and hence it is a Poisson subvariety.  (Indeed, it is a symplectic leaf.)  The vector field $\cvf{z}$ is clearly a Poisson vector field, but it does not preserve $\W$.  Hence $\W$ is a Poisson subscheme which is not strong.\qed
\end{example}

\begin{lemma}\label{lem:strg-subvar}
The reduced scheme $\X_{red}$ of a Poisson scheme $(\X,\ps)$ and all its irreducible components are strong Poisson subvarieties.  Similarly, the singular locus $\X_{sing}$ of $\X$ is a strong Poisson subscheme.
\end{lemma}

\begin{proof}
Lemma 1.1 in \cite{Polishchuk1997} shows that the ideals defining the reduced subscheme and its irreducible components are preserved by all vector fields.  In particular, they are preserved by $\sPois{\ps}$.

Similarly, when $\X$ has dimension $n$, the ideal $\sI{\X_{sing}}$ of $\X_{sing}$ is the annihilator of $\forms[n+1]{\X}$.  If $Z \in \tshf{X}$ is a vector field, $s \in \forms[n+1]{X}$, and $f \in \sI{\X_{sing}}$, we have
$$
0 = \lie{Z}(fs) = Z(f)s + f\lie{Z}s = Z(f)s,
$$
so that $Z(f) \in \sI{\X_{sing}}$.  It follows that $\sI{\X_{sing}}$ is preserved by all vector fields.  In particular, it is preserved by $\sPois{\ps}$.
\end{proof}

\begin{lemma}\label{lem:strong-pairs}
If $\Y_1$ and $\Y_2$ are Poisson subschemes of $(\X,\ps)$, then so are $\Y_1\cap \Y_2$ and $\Y_1 \cup \Y_2$.  Furthermore, if $\Y_1$ and $\Y_2$ are strong Poisson subschemes, then $\Y_1 \cap \Y_2$ and $\Y_1 \cup \Y_2$ are also strong.
\end{lemma}

\begin{proof}
Simply notice that if $Z \in \tshf{X}$ is a vector field preserving the ideals $\sI{\Y_1}$ and $\sI{\Y_2}$ defining $\Y_1$ and $\Y_2$, it also preserves $\sI{\Y_1} + \sI{\Y_2}$ and $\sI{\Y_1} \cap \sI{\Y_2}$.  Apply this observation to Hamiltonian and Poisson vector fields, respectively.
\end{proof}

Perhaps the most important examples of Poisson subschemes are the degeneracy loci of the Poisson structure:
\begin{definition}
Let $(\X,\ps)$ be a Poisson scheme.  The \defn{$2k^{th}$ degeneracy locus of $\ps$} is the locus $\D[2k]{\ps} \subset \X$ where the morphism $\pss : \forms[1]{X} \to \tshf{X}$ has rank $\le 2k$.  It is the closed subscheme whose ideal sheaf is the image of the morphism
$$
\xymatrix{
\forms[2k+2]{\X} \ar[r]^{\ps^{k+1}} & \strc{X},
}
$$
where
$$
\ps^{k+1} = \underbrace{\ps \wedge \cdots \wedge \ps}_{k+1\mathrm{\ times}} \in \cohlgy[0]{\X,\der[2k+2]{\X}}.
$$
\end{definition}

\begin{remark}
The rank of $\ps$ is always even because $\ps$ is skew-symmetric.  In the analytic topology, the subscheme $\D[2k]{\ps}$ is simply the union of all symplectic leaves of $\ps$ having dimension $\le 2k$.\qed
\end{remark}

\begin{remark}
The section $\ps^{k+1}$ has the property that for $f_1,\ldots,f_{2k+2} \in \strc{X}$, the element
$$
\ps^{k+1}(df_1\wedge\cdots\wedge df_{2k+2}) \in \strc{X}
$$
is equal to the Pfaffian of the matrix $\{f_i,f_j\}_{i,j \le 2k+2}$ of Poisson brackets.  In particular, this definition is appropriate even when $\X$ is singular. However, it is important to note that when $\X$ is singular, it is possible that $\D[2k]{\ps}$ is non-empty although $\ps^{k+1}$ is non-vanishing as a section of $\der[2k+2]{\X}$.\qed
\end{remark}

\begin{example}  Consider the Poisson structure of \autoref{ex:cone}, where $\X$ is the cone $uw=v^2$ in $\Aff{3}$, and the Poisson structure is given by
$$
\ps = 2u\cvf{u}\wedge\cvf{v} + 4v\cvf{u}\wedge\cvf{w} + 2w\cvf{v}\wedge\cvf{w} .
$$
In this case, $\ps$ gives a trivialization of the $\strc{X}$-module $\der[2]{\X}$, because $\X$ is normal, $\der[2]{\X}$ is reflexive, and $\ps$ is non-vanishing on the smooth locus.  Therefore, $\ps$ is a non-vanishing section of $\der[2]{\X}$.  However, the ideal defining $\D[0]{\ps}$ is generated by $u,v$ and $w$.  Hence, $\D[0]{\ps}$ is the singular point of the cone, which is the origin in $\Aff{3}$. \qed 
\end{example}

The definition of the degeneracy loci is compatible with Poisson subschemes:
\begin{lemma}
If $(\X,\ps)$ is a Poisson scheme, and $\Y$ is a Poisson subscheme, then $\D[2k]{\ps} \cap \Y = \D[2k]{\ps|_\Y}$ as schemes.
\end{lemma}

\begin{proof}
The diagram
$$
\xymatrix{
\forms[2k+2]{\X} \ar[d] \ar[r]^{\ps^{k+1}} & \strc{\X} \ar[d]\\
\forms[2k+2]{\Y} \ar[r]^{\ps|_\Y^{k+1}} & \strc{\Y}
}
$$
commutes.  Since the map $\forms[2k+2]{\X} \to \forms[2k+2]{\Y}$ is surjective, we see that the image of $\ps|_\Y^{k+1}$ is the same as the image of $\ps^{k+1}(\forms[2k+2]{\X})$ in $\strc{\Y}$, as required.
\end{proof}

Polishchuk showed~\cite[Corollary 2.3]{Polishchuk1997} that the degeneracy loci are Poisson subschemes.  In fact, as he notes, they are strong Poisson subschemes:
\begin{proposition}\label{prop:degen-strong}Let $(\X,\ps)$ be a Poisson scheme.  Then for $0 \le 2k \le \dim \X$, the degeneracy loci $\D[2k]{\ps}$ are strong Poisson subschemes of $\X$.
\end{proposition}

\begin{proof}
Let $\sI{}$ be the ideal sheaf of $\D[2k]{\ps}$.  We have the exact sequence
$$
\xymatrix{
\forms[2k+2]{X} \ar[r]^-{\ps^{k+1}} &  \sI{} \ar[r] &0,
}
$$
so it suffices to show that for every $\alpha \in \forms[2k+2]{X}$ and every Poisson vector field $Z \in \sPois{\ps}$, the function $\lie{Z} (\ps^{k+1}(\alpha))$ lies in the image of $\ps^{k+1}$.  We compute
\begin{align*}
\lie{Z} (\ps^{k+1}(\alpha)) &= (\lie{Z}\ps^{k+1})(\alpha) + \ps^{k+1}(\lie{Z}\alpha) \\
&= \ps^{k+1}(\lie{Z}\alpha),
\end{align*}
since $\lie{Z}\ps= 0$, showing that $Z(f)\in \sI{{\D[2k]{\ps}}}$, as required.
\end{proof}

It will often be convenient to consider the degeneracy loci and the singular loci in tandem.  We therefore make the following definition:
\begin{definition}\label{defn:sing-pois}
Let $(\X,\ps)$ be a Poisson scheme, and suppose that the maximal rank of $\ps$ is equal to $2k$.  The \defn{Poisson singular locus} of $(\X,\ps)$ is the strong Poisson subscheme defined by
$$
\Sing{\X,\ps} = \X_{sing} \cup \D[2k-2]{\ps}. 
%\subset \X
$$
%which is the union of the singular subscheme of $\X$ and the degeneracy locus of $\ps$.  
We say that $(\X,\ps)$ is \defn{regular} if $\X$ is smooth and $\ps$ has constant rank, ie., $\Sing{\X,\ps} = \varnothing$.
\end{definition}

\section{Geometry of Poisson modules}
\label{sec:mod-geom}
\subsection{Poisson modules}

Let $(\X,\ps)$ be a Poisson scheme, and let $\sE$ be a sheaf of $\strc{X}$-modules.  A \defn{Poisson connection} on $\sE$ is a $\complex$-linear morphism of sheaves $\nabla : \sE \to \der[1]{\X}(\sE)$ satisfying the Leibniz rule
$$
\nabla(fs) = -\pss(df)\otimes s + f\nabla s
$$
for all $f \in \strc{X}$ and $s \in \sE$.  Recall that $\der[1]{\X}(\sE) = \sHom[]{{\forms[1]{X},\sE}}$, so if $\X$ is smooth or $\sE$ is locally-free, we have $\der[1]{\X}(\sE) \cong \tshf{X} \otimes \sE$.

%\begin{remark}\marco{this remark doesn't help - perhaps it should be removed.  The only reason is that you want $\{f,s\}=\nabla_{df}s$ to satisfy the right Leibniz property.  We could just not explain the sign.}
%%The minus sign in the Leibniz rule is inserted 
%%so that the operation $\{f,s\}=\nabla_{df} s$ satisfies $\{f,gs\}=\{f,g\}s + g\{f,s\}$.
%due to the skew-symmetry: $(\pss)^* = - \pss : \forms[1]{X} \to \tshf{X}$ when $\X$ is smooth.
%\end{remark}
\begin{remark}\label{rem:cnxn-jet} Let $\sJet[1]{\sE}$ be the sheaf of one-jets of sections of $\sE$.  Then a Poisson connection is an $\strc{X}$-linear map $\nabla : \sJet[1]{\sE} \to \der[1]{\X}(\sE)$ making the following diagram commute:
$$
\begin{CD}
\forms[1]{\X} \otimes \sE @>>> \sJet[1]{\sE} \\
@V{-\pss \otimes \id{\sE}}VV		@VV{\nabla}V \\
\tshf{\X} \otimes \sE @>>>		\der[1]{\X}(\sE)
\end{CD}
$$
In other words, a Poisson connection is a first-order differential operator on $\sE$ whose symbol is $-\pss \otimes \id{\sE}$.  
\qed
\end{remark}

Using a Poisson connection, we may differentiate a section of $\sE$ along a one-form: if $\alpha \in \forms[1]{X}$, we set
$$
\nabla_\alpha s = (\nabla s)(\alpha).
$$
The Poisson connection $\nabla$ is \defn{flat} if
$$
\nabla_{d\{f,g\}}s = (\nabla_{df}\nabla_{dg} - \nabla_{dg} \nabla_{df} )s 
$$
for all $f,g\in\strc{X}$ and $s \in \sE$.  A \defn{Poisson module} is a sheaf of $\strc{X}$-modules equipped with a flat Poisson connection.

\begin{remark}
For convenience, we will refer to a Poisson module $(\sE,\nabla)$ as \defn{locally-free} (resp., \defn{invertible}) if the underlying $\strc{X}$-module is locally-free (resp., invertible).  In particular, we warn the reader that a locally-free Poisson module need not be locally isomorphic to a direct sum of copies of $\strc{X}$ as a Poisson module.\qed
\end{remark}

There are two important examples of Poisson modules:

\begin{example}\rm The composite morphism
$$
-\pss d : \strc{X} \to \tshf{X} = \der[1]{\X}(\strc{X})
$$
taking a function to minus its Hamiltonian vector field defines a Poisson connection on $\strc{X}$ by the Leibniz rule for $d$.  This connection is flat because of the Jacobi identity for the Poisson bracket.\qed
\end{example}

\begin{example}\label{ex:can-mod}\rm
When $\X$ is smooth of pure dimension $n$, there is a natural Poisson module structure on the canonical sheaf $\can_\X = \forms[n]{X}$.  For $\alpha \in \forms[1]{\X}$ and $\mu \in \can_\X$, the connection is defined by the formula
$$
\nabla_{\alpha} \mu = - \alpha \wedge d \ips \mu,
$$
where $\ips$ denotes interior contraction by $\ps$.  For a function $f \in \strc{X}$, we have the identity
$$
\nabla_{df}\mu = -\lie{\pss(df)}\mu.
$$
This module is variously referred to as the \defn{modular representation}, the \defn{Evens--Lu--Weinstein module} or the \defn{canonical module}~\cite{Evens1999,Polishchuk1997}.  As we shall see, the geometry of the canonical module is intimately connected with the degeneracy loci of $\ps$. \qed
\end{example}

As usual, we can build new modules out of old:
\begin{example}\rm
If $\sE$ and $\sE'$ are sheaves of $\strc{X}$-modules with Poisson connections $\nabla$ and $\nabla'$, the sheaves $\sE \oplus \sE'$, $\sE \otimes \sE'$ and $\sHom[\strc{X}]{\sE,\sE'}$ inherit Poisson connections in the obvious way, as do the exterior and symmetric powers $\ext{\sE}$, $\sym{\sE}$.  When $\nabla$ and $\nabla'$ are flat, so are the newly constructed connections.

Similarly, if $\sL$ is an invertible sheaf with an isomorphism $\sL^{\otimes k} \cong \sE$, there is a unique Poisson connection on $\sL$ which induces the given one on $\sE$.  In other words, $\nabla$ has a well-defined $k^{th}$-root, which is flat if and only if $\nabla$ itself is.\qed
\end{example}

%\begin{example}\label{ex:picard}
%Suppose that $\X$ is smooth and proper of dimension $n$ with Picard group $\Pic{\X} \cong \integer$, and suppose that the canonical sheaf $\can_\X = \forms[n]{X}$ has degree $k \ne 0$.  For example, $\X$ could be a projective space.  Then any generator of $\Pic{\X}$ inherits a natural Poisson connection as the $k^{th}$ root of $\can_\X$.  It follows that all invertible sheaves on $\X$ are Poisson modules in a natural way.\qed
%\brent{I was thinking about removing this example}\marco{I have no opinion.}
%\end{example}
If $\sL$ is an invertible sheaf equipped with a Poisson connection $\nabla$, and $s \in \sL$ is a local trivialization, we obtain a unique vector field $Z \in \tshf{X}$ such that
$$
\nabla s = Z \otimes s.
$$
This vector field is called the \defn{connection vector field} for the trivialization $s$.  Then $\nabla$ is flat if and only if $Z$ is a Poisson vector field.

\begin{definition}
Let $(\X,\ps)$ be a Poisson scheme.  The connection vector field for a local trivialization $\phi \in \can_\X$ of the canonical module is called the \defn{modular vector field} associated to $\phi$.
\end{definition}

Given two Poisson connections $\nabla,\nabla'$ on $\sE$, the difference $\nabla-\nabla'$ is $\strc{X}$-linear, defining a global morphism $\nabla-\nabla' \in \Hom[\X]{{\forms[1]{\X}\otimes \sE,\sE}}$.  We thus have
\begin{lemma}
The space of Poisson connections on $\sE$ is an affine space modelled on $\Hom[\X]{{\forms[1]{\X}\otimes \sE,\sE}} = \cohlgy[0]{\X,\der[1]{\X}(\sEnd{\sE})}$, provided it is nonempty.
\end{lemma}

Let $\sF = \sImg{\pss}$ and $\sN{} = \der[1]{\X}/\sF$, and consider the exact sequence
$$
0 \to \sF \to \der[1]{\X} \to \sN{} \to 0.
$$
If $\sE$ is locally free, then $\der[1]{\X}(\sE) \cong \der[1]{\X}\otimes \sE$, and the sequence
$$
0 \to \sF\otimes \sE \to \der[1]{\X}(\sE) \to \sN{}\otimes \sE \to 0
$$
is also exact.  Given a Poisson connection $\nabla$ on $\sE$, the composite morphism
$$
\sE \stackrel{\nabla}{\to} \der[1]{\X}(\sE) \to \sN{} \otimes \sE
$$
is actually $\strc{X}$-linear, and hence defines a \defn{normal Higgs field}
$$
\Phi_\nabla \in \cohlgy[0]{\X,\sN{}\otimes \sEnd{\sE}},
$$
which measures the failure of the connection vector fields for $\nabla$ to be tangent to the symplectic leaves of $\ps$.

\begin{definition}\label{def:adapt}
The Poisson connection $\nabla$ on the locally-free sheaf $\sE$ is \defn{adapted} if its normal Higgs field $\Phi_\nabla$ vanishes.  Hence $\nabla$ is adapted exactly when it is induced by a $\complex$-linear morphism
$$
\sE \to \sF\otimes \sE,
$$
which may be seen as a Poisson connection on the symplectic leaves of $\ps$.
\end{definition}
Since a Poisson connection on a symplectic manifold is the same thing as a flat connection in the ordinary sense, we may think of an adapted Poisson module as defining a flat partial connection along the symplectic leaves.  We shall see that modules regularly fail to be adapted, and that this phenomenon has important consequences for the geometry of the degeneracy loci.

\begin{example}\label{ex:adapt}
Consider once again the the Poisson structure $\ps =\cvf{x}\wedge\cvf{y}$ on $\X = \Aff{3}=\Spec(\complex[x,y,z])$.  Any vector field $Z$ induces a Poisson connection on $\strc{X}$ by the formula
$$
\nabla f = -\pss(df) + fZ \in \der[1]{\X}.
$$
Then $\sF = \sImg{\pss}$ is generated by $\cvf{x}$ and $\cvf{y}$, and the normal Higgs field $\Phi_\nabla$ is just the class of $Z$ in the normal bundle $\sN{}=\der[1]{\X}/\sF$ to the symplectic leaves.  Hence the connection is adapted if and only if $Z$ is an $\strc{X}$-linear combination of $\cvf{x}$ and $\cvf{y}$.  In particular, the connection associated to the Poisson vector field $Z = \cvf{z}$ is not adapted, even though it is flat.\qed
\end{example}

\begin{example}\label{ex:can-adapt}
Consider the Poisson structure  $\ps = x \cvf{x}\wedge\cvf{y}$ on $\Aff{2} = \Spec(\complex[x,y])$, which vanishes on the $y$-axis $\Y \subset \Aff{2}$.  Its image sheaf $\sF$ is generated by $x\cvf{x}$ and $x\cvf{y}$, yielding a torsion normal sheaf $\sN{}$ freely generated over $\strc{\Y}$ by $\cvf{x}|_\Y$ and $\cvf{y}|_\Y$.  Using the formulae in \autoref{sec:mod-res}, the Poisson module structure on $\can_{\Aff{2}}$ is readily computed.  We have
$$
\nabla (dx\wedge dy) = \cvf{y} \otimes (dx \wedge dy),
$$
 and hence the normal Higgs field for the canonical module is
 $$
 \Phi_\nabla = \cvf{y}|_\Y \ne 0.
 $$
 Therefore $\can_{\Aff{2}}$ is not adapted.  However, its restriction to the open dense set $\Aff{2} \setminus \Y$ is adapted.\qed
\end{example}

The interaction between Poisson modules and Poisson subschemes is an important, but subtle, aspect of the theory.  This interaction will play an important role in our our proof of \autoref{thm:bondal}, allowing us to ``tunnel down'' to the lower-rank strata of the Poisson structure.  In general, it is not true that a Poisson module $(\sE,\nabla)$ on $\X$ restricts to a Poisson module  on a closed Poisson subscheme $\Y \subset \X$.  Indeed, if $\sI{\Y}$ is the ideal sheaf of $\Y$ in $\X$, the conormal sheaf is $\sI{\Y}/\sI{\Y}^2$, and we have an exact sequence
$$
0 \to \der[1]{\Y}(\sE) \to \der[1]{\X}(\sE)|_\Y \to\sHom[{\strc{\Y}}]{{\sI{\Y}/\sI{\Y}^2,\sE|_\Y}}.
$$
Similar to our above considerations, we obtain from the connection a normal Higgs field along $\Y$:
$$
\Phi^\Y_\nabla \in \Hom[\Y]{\sI{\Y}/\sI{\Y}^2,\sEnd{\sE}|_\Y}.
$$
This Higgs field vanishes identically if and only if $(\sE|_\Y,\nabla|_\Y)$ defines a Poisson connection for $\ps|_\Y$.  When $\sE$ is locally free, $\Phi^\Y_\nabla$ is a section of $\sN{\Y} \otimes \sEnd{\sE}$, where $\sN{\Y} = (\sI{\Y}/\sI{\Y}^2)^\vee$ is the normal sheaf of $\Y$ in $\X$.  

A useful aspect of strong Poisson subschemes is that they behave well with respect to Poisson modules:
\begin{lemma}\label{prop:strong-module}
Let $(\X,\ps)$ be a Poisson scheme and let $(\sL,\nabla)$ be an invertible Poisson module.  Then if $\Y \subset \X$ is any strong Poisson subscheme of $\X$, the restriction $(\sL|_\Y,\nabla|_\Y)$ is a Poisson module on $\Y$ with respect to the induced Poisson structure.
\end{lemma}

\begin{proof}
Choosing a local trivialization $s$ of $\sL$, we have
$$
\nabla s = Z \otimes s,
$$
where $Z$ is a Poisson vector field.  Since $\Y$ is a strong Poisson subvariety, we have $Z(\sI{\Y})\subset \sI{\Y}$.  Hence the image of $Z \otimes s$ in $\sHom[{\strc{\Y}}]{{\sI{\Y}/\sI{\Y}^2,\sL|_\Y}}$ is zero and $Z \otimes s \in \der[1]{\Y}(\sL)$, as required.
\end{proof}

\subsection{The distribution induced by a Poisson module}
\label{sec:mod-dist}
%Suppose that $(\X,\ps)$ is a Poisson scheme, and that $\sE$ is a sheaf of $\strc{X}$-modules on $\X$ equipped with the Poisson connection $\nabla$.  If $s \in \sE$ and $\alpha \in \sE^\vee$, we obtain a vector field $\alpha(\nabla s)$ by the pairing of $\sE^\vee$ with $\der[1]{\X}(\sE)$.  The $\strc{X}$-submodule generated by such vector fields forms a distribution\marco{kill this?}
%$$
%\Mod{\nabla} = \set{ \alpha(\nabla s) }{ s \in \sE{\rm\ and\ }\alpha \in \sE^\vee } \subset \der[1]{\X}.
%$$
%Equivalently, $\Mod{\nabla}$ is the image of the natural morphism
%\begin{align}
%\xymatrix{
%\sJet[1]{\sE}\otimes \sE^\vee \ar[r]^-{\nabla \otimes 1} & \der[1]{\X}(\sE) \otimes \sE^\vee \ar[r] & \der[1]{\X}
%}\label{eqn:mod-dist}
%\end{align}
%induced by $\nabla$; see \autoref{rem:cnxn-jet}.

The choice of a Poisson module structure $\nabla$ on the invertible sheaf $\sL$ induces a Poisson bracket $\cbrac{\cdot,\cdot}_\nabla$ on the section ring $\bigoplus_{n \ge 0} \sL^n$, rendering the total space of $\sL^\vee$ into a Poisson scheme~\cite[Corollary 5.3]{Polishchuk1997}.  The bracket is uniquely determined by the identities
\begin{align*}
\cbrac{f,g}_\nabla &= \cbrac{f,g} \\
\cbrac{f,s}_\nabla &= -\nabla_{df} s,
\end{align*}
for $f \in \strc{X}$ and $s \in \sL$, and is compatible with the grading in the sense that $\cbrac{\sL^i,\sL^j}_\nabla\subset \sL^{i+j}$ for all $i,j \ge 0$.  This Poisson structure has an interpretation in terms of the module of differential operators on $\sL$, which we now describe.

Let $\forms[1]{\sA_\sL} = \sJet[1]{\sL}\otimes \sL^\vee$, so that $\sA_\sL = (\forms[1]{\sA_\sL})^\vee$ is the sheaf of differential operators of degree $\le 1$ on the $\strc{X}$-module $\sL$.  Recall that $\sA_\sL$ is known as the \defn{Atiyah algebroid} of $\sL$, and the jet sequence for $\sL$ leads to the exact sequence
\begin{equation}\label{exactjet}
\xymatrix{
0 \ar[r] & \strc{X} \ar[r] & \sA_\sL \ar[r]^a & \der[1]{\X} \ar[r] & 0,
}
\end{equation}
where the map $\strc{X} \to \sA_\sL$ takes a function $f$ to the operator of multiplication by $f$, and a differential operator $D \in \sA_\sL$ is taken to the derivation $a(D)(f)=[D,f]$.  

The bracket $\cbrac{\cdot,\cdot}_\nabla$ of sections of $\sL$ depends only on their one-jet, and defines a bidifferential operator
$$
\ps_\nabla \in \cohlgy[0]{\X,\sA^2_\sL},
$$
where we use $\sA^k_\sL= (\wedge^k \forms[1]{\sA_\sL})^\vee$ to denote the sheaf of totally skew $k$-differential operators on $\sL$.  The section $\ps_\nabla$ will play an important role in our understanding of adaptedness in \autoref{sec:resolve}.  Its image under the natural map $\sA^2_\sL \to \der[2]{\X}$ is the Poisson structure $\ps$.

The section $\ps_\nabla$ makes $\sA_\sL$ into a triangular Lie bialgebroid in the sense of Mackenzie and Xu~\cite{Mackenzie1994}.  In particular, the sheaf $\forms[1]{\sA_\sL}$ inherits a Lie bracket, and the composition
$$
a^*:\xymatrix{
\forms[1]{\sA_\sL} \ar[r]^{\pss_\nabla} & \sA_\sL \ar[r]^a & \der[1]{\X}
}
$$
is a morphism of sheaves of Lie algebras.  
%This composition is readily identified with the map $\sJet[1]{\sL} \otimes \sL^\vee \to \der[1]{\X}$ in \eqref{eqn:mod-dist} whose image is the distribution $\Mod{\nabla}$.  
Moreover, the following diagram commutes:
$$
\xymatrix{
\forms[1]{\sA_\sL} \ar[r]^{\pss_\nabla} & \sA_\sL \ar[d] \\
\forms[1]{\X} \ar[u] \ar[r]^\pss & \der[1]{\X} 
}
$$
Combining these facts immediately gives us the following
\begin{proposition}
Suppose that $(\sL,\nabla)$ is an invertible Poisson module on $(\X,\ps)$.  Then the distribution $a^*(\forms[1]{\sA_\sL}) \subset \tshf{X}$ is involutive for the Lie bracket, and contains $\sImg{\pss}$.
\end{proposition}

\begin{remark}
A more down-to-earth proof is as follows: given a local trivialization $s \in \sL$, we obtain the connection vector field $Z \in \der[1]{\X}$, and it is clear that $a^*(\forms[1]{\sA_\sL})$ is generated over $\strc{\X}$ by $Z$ and $\sHam{\ps}$.  Since the connection is flat, $Z$ is Poisson and hence $\lie{Z}$ preserves $\sHam{\ps}$.  Therefore the distribution is involutive.\qed
\end{remark}

\begin{remark}
It follows that every invertible Poisson module induces a singular foliation of the analytic space associated to $\X$.  Since the canonical sheaf $\can_\X$ carries a natural Poisson module structure $\nabla^{\can_\X}$, every Poisson manifold comes equipped with \emph{two} natural singular foliations: there is the usual foliation by symplectic leaves, and then there is the image of $a^*$ for the canonical module, which we call the \defn{modular foliation}.  While we always have $\sImg{\pss}~\subset~\sImg{a^*}$, the inclusion is, in general, strict as we saw in \autoref{ex:can-adapt}.  It would be interesting to study this secondary foliation in greater detail. 
%We are not aware of any reference to this secondary foliation in the literature devoted to the modular representation.\brent{We could replace this sentence with something a little softer, like ``It would be interesting to study this secondary foliation in greater detail.''}
\qed
\end{remark}

\subsection{Residues of invertible Poisson modules}
\label{sec:defn-res}
In this section, we define the residues of a Poisson line bundle.  These residues are multiderivations supported on the degeneracy loci, and they encode features of the connection as well as the degeneracy loci themselves.

\begin{proposition}
Let $(\X,\ps)$ be a Poisson scheme, with an invertible Poisson module $(\sL,\nabla)$.  Then the morphism defined by the composition
$$
\xymatrix{\sJet[1]{\sL}\ar[r]^-{\nabla} & \der[1]{\X}(\sL)\ar[r]^-{\ps^k} & \der[2k+1]{\X}(\sL)}
$$
%$$
%\begin{CD}
%\sL @>{\ps^k}>> \der[1]{\X}(\sL) @>>> \der[2k+1]{\X}(\sL) 
%\end{CD}
%$$
descends, upon restriction to $\D[2k]{\ps}$, to a morphism $\sL\to\der[2k+1]{{\D[2k]{\ps}}}(\sL)$, defining a multiderivation 
$$
\res{k}\nabla \in \cohlgy[0]{\D[2k]{\ps},\der[2k+1]{{\D[2k]{\ps}}}}.
$$
\end{proposition}
\begin{proof}
Let $\Y=\D[2k]{\ps}$. We abuse notation and denote the restricted Poisson structure by $\ps \in \der[2]{\Y}$.  Since $\Y$ is a strong Poisson subscheme, \autoref{prop:strong-module} guarantees that $\sL|_\Y$ is a Poisson module, and so the connection on $\sL$ restricts to a morphism $\nabla|_\Y : \sJet[1][\Y]{\sL} \to \der[1]{\Y}(\sL)$.  Consider the commutative diagram
$$
\xymatrix{
%0 \ar[d]					&			& \\
\forms[1]{\Y}\otimes \sL|_\Y \ar[d] \ar[dr]^{-\pss\otimes \id{\sL}} & & \\
\sJet[1][\Y]{\sL} \ar[d] \ar[r]^-{\nabla|_\Y} & \der[1]{\Y}(\sL) \ar[r]^-{\ps^k} & \der[2k+1]{\Y}(\sL) \\
\sL|_\Y \ar@{-->}[urr]& &\\
%\sL|_\Y \ar[d] \ar@{-->}[urr]& &\\
%0  & &\\
}
$$
By exactness of the jet sequence, the composition $\ps^k\circ\nabla|_\Y$ will descend to give the dashed arrow we seek provided that $\ps^k \wedge \pss(\xi) = 0$ for all $\xi \in \forms[1]{\Y}$. But using the contraction $i:\forms[1]{\Y} \otimes \der{\Y} \to \der[\bullet-1]{\Y}$, we compute
$$
\pss(\xi)\wedge \ps^k = (i_\xi\ps)\wedge\ps^k =  \tfrac{1}{k+1} i_\xi(\ps^{k+1}),
$$
which vanishes identically on $\Y$ since $\ps^{k+1}|_\Y = 0$.
\end{proof}

\begin{definition}\label{def:mod-tensor}
The section
$$
\res{k}{\nabla} \in \cohlgy[0]{\D[2k]{\ps},\der[2k+1]{{\D[2k]{\ps}}} }
$$
defined by an invertible Poisson module $(\sL,\nabla)$ is called the \defn{$k^{th}$ residue of $\nabla$}.  The $k^{th}$ residue of the canonical module $\omega_\X$ is called the \defn{$k^{th}$ modular residue of $\ps$}, and denoted
$$
\mres{k}{\ps} \in \cohlgy[0]{\D[2k]{\ps},\der[2k+1]{{\D[2k]{\ps}}} }.
$$ 
\end{definition}

\begin{remark}
The modular residue is so named because it locally has the form $Z \wedge \ps^k$, where $Z$ is the modular vector field of $\ps$ with respect to a local trivialization.  In \autoref{sec:elliptic}, we will see a family of examples where the modular residues are all non-vanishing, giving trivializations of $\der[2k+1]{{\D[2k]{\ps}}}$.\qed
\end{remark}

\begin{remark}
By the argument above, any Poisson connection on the $\strc{X}$-module $\sE$ defines a section of $\der[2k+1]{\X}(\sEnd{\sE})|_{\D[2k]{\ps}}$.  Invertibility and flatness imply that this section actually defines a multi-derivation on $\D[2k]{\ps}$.\qed
\end{remark}

\begin{remark}\label{rem:higgs-res}
The residue of $\nabla$ can also be described in terms of the normal Higgs field $\Phi_\nabla \in \cohlgy[0]{\X,\sN{}}$: since the rank of $\ps$ on $\D[2k]{\ps}$ is $\le 2k$, the exterior product of $\Phi_\nabla$ and $\ps^k$ gives a well-defined section of $\der[2k+1]{\X}|_{\D[2k]{\ps}}$, and we have
$$
i_*\res{k}{\nabla} = \Phi_\nabla|_{\D[2k]{\ps}} \wedge \ps|^k_{\D[2k]{\ps}}  \in \der[2k+1]{\X}|_{\D[2k]{\ps}},
$$
where $i : \D[2k]{\ps} \to \X$ is the inclusion.  If $Z \in \der[1]{\X}$ is the connection vector field associated to a local trivialization of the module, we have
$$
i_*\res{k}{\nabla} = (Z \wedge \ps^k)|_{\D[2k]{\ps}}
$$
on the domain of $Z$.\qed
\end{remark}

\section{The degeneracy loci of a Poisson module}
\label{sec:resolve}
In this section, we suppose that $(\X,\ps)$ is a smooth, connected Poisson scheme.
Recall from \autoref{sec:mod-dist} that an invertible Poisson module $(\sL,\nabla)$ gives rise to a canonical section $\ps_\nabla \in \cohlgy[0]{\X,\sA^2_\sL}$.  We now study the degeneracy loci $\AD[2k]{\nabla}$ of this section.  These loci provide a new family of Poisson subschemes with the property that the Poisson module is flat along their $2k$-dimensional symplectic leaves.
%Intuitively speaking, these loci are the maximal subschemes on which $\nabla$ is adapted to the symplectic leaves of a given dimension.  
%\begin{definition}\label{lem:mod-form-rank}
%Let $(\sL,\nabla)$ be an invertible Poisson module.  The degeneracy scheme
%$$
%\AD[2k]{\nabla} = (\ps_\nabla^{k+1})^{-1}(0) \subset \X
%$$
%on which $\ps_\nabla$ has rank $\le 2k$ is called the \defn{$2k$-adapted locus of $\nabla$}.
%\end{definition}

\begin{remark}\label{rem:adapt-lb}
One could also think of $\ps_\nabla$ as the Poisson structure induced on the total space of $\sL^\vee$.  If $\Y$ is the principal $\complex^*$-bundle obtained by removing the zero section of $\sL^\vee$, then $\AD[2k]{\nabla}$ is the image of the $2k^{th}$ degeneracy locus of the Poisson structure on $\Y$ under the projection.\qed
\end{remark}

The following proposition shows that the degeneracy loci of $\ps$ and $\ps_\nabla$ are compatible, and explains the relationship between $\ps_\nabla$ and the adaptedness of the module in the sense of \autoref{def:adapt}:
\begin{proposition}\label{prop:adapt-locus}
The degeneracy locus $\AD[2k]{\nabla}$ of an invertible Poisson module $(\sL,\nabla)$ is the zero locus	of the section
$$
i_* \res{k}{\nabla} \in \cohlgy[0]{\D[2k]{\ps},\der[2k+1]{\X}|_{\D[2k]{\ps}}}
$$
where $i : \D[2k]{\ps} \to \X$ is the embedding.  In particular, we have inclusions
$$
\D[2k-2]{\ps} \subset \AD[2k]{\nabla} \subset \D[2k]{\ps}.
$$
Moreover, $\nabla$ restricts to a Poisson module on $\AD[2k]{\nabla}$ which is adapted on the open set $\AD[2k]{\nabla} \setminus \D[2k-2]{\ps}$.
\end{proposition}

\begin{proof}
Because of the exact sequence~\eqref{exactjet}, we have the extension
$$
0 \to \der[2k+1]{\X} \to {\sA^{2k+2}_\sL} \to \der[2k+2]{\X} \to 0.
$$
The image of $\ps_\nabla^{k+1}$ in $\der[2k+2]{\X}$ is simply $\ps^{k+1}$.  Hence the zero scheme of $\ps_\nabla^{k+1}$ is contained in $\D[2k]{\ps}$, and is given by the vanishing of the section
$$
s = \ps_\nabla^{k+1}|_{\D[2k]{\ps}} \in \der[2k+1]{\X}|_{\D[2k]{\ps}}.
$$
One readily checks using the definition of $\ps_\nabla$ that $s$ has the form
\begin{align}
s = Z \wedge \ps^{k}|_{\D[2k]{\ps}}, \label{eqn:ad-res}
\end{align}
where $Z$ is a connection vector field for $\nabla$ with respect to any local trivialization.  But then $s = i_*\res{k}{\nabla}$ by \autoref{rem:higgs-res}, and so $\AD[2k]{\nabla}$ is the zero scheme of $i_*\res{k}{\nabla}$, as claimed.  Since $s$ has a factor $\ps^k$, we have $\D[2k-2]{\ps} \subset \AD[2k]{\nabla}$.

We wish to show that the zero locus of $s$ is a Poisson subscheme.  Given a function $f \in \strc{X}$, we have
\begin{align*}
\lie{\pss(df)}(\ps^k\wedge Z) &= (\lie{\pss(df)} \ps^k) \wedge Z + \ps^k \wedge (\lie{\pss(df)}Z )\\
&= - \ps^k \wedge \lie{Z} \pss(df) \\
&= - \ps^k \wedge \pss(\lie{Z}df) \\
&= 0 \mathrm{\ \ mod\ }\sI{\D[2k]{\ps}},
\end{align*}
where we have used the identities $\lie{\pss(df)}\ps = 0$ and $\lie{Z}\ps = 0$.  It follows that the ideal defining $\AD[2k]{\nabla}$ is preserved by all Hamiltonian vector fields, and hence $\AD[2k]{\nabla}$ is a Poisson subscheme.

To see that the module is adapted on $\AD[2k]{\nabla} \setminus \D[2k-2]{\ps}$, simply notice that if at some point $x \in \D[2k]{\ps}$ we have $\ps^k_x \ne 0$, then $Z_x$ lies in the image of $\pss_x$ if and only if $Z_x \wedge \ps^k_x = 0$, ie. $x \in \AD[2k]{\nabla}$.
\end{proof}

\begin{example}\label{ex:euler-planes}
Consider the Poisson structure $\ps = (x\cvf{x} + y\cvf{y})\wedge \cvf{z}$ on the affine space $\Aff{3} = \Spec(\complex[x,y,z])$.  Denote by $\Z \subset \Aff{3}$ the $z$-axis, given by $x=y=0$.  The degeneracy loci are $\D[2]{\ps} = \Aff{3}$ and $\D[0]{\ps} = \Z$.  Every plane $\W$ containing $\Z$ is a Poisson subvariety, and $\W \setminus \Z$ is a symplectic leaf.

The vector field
$$
U = x\cvf{x}
$$
is Poisson.  (Indeed, any $\complex$-linear combination of $x \cvf{x}, y\cvf{x}, x\cvf{y}$ and $y\cvf{y}$ is Poisson.)  Hence $U$ defines a Poisson module structure on $\strc{X}$ by the formula
$$
\nabla f = -\pss(df) \otimes 1 + U \otimes f.
$$
The degeneracy locus $\AD[2]{\nabla}$ is given by the vanishing of the tensor
$$
U \wedge \ps = xy\cvf{y}\wedge\cvf{z},
$$
and hence  $\AD[2]{\nabla}$ is the union of the $x=0$ and $y=0$ planes, which are indeed Poisson subschemes.  We therefore see that the inclusions
$$
\D[0]{\ps} \subset \AD[2]{\nabla} \subset \D[2]{\ps}
$$
are strict in this case.\qed
\end{example}

The expected dimensions and fundamental classes of the degeneracy loci of skew morphisms are described by the following result in intersection theory. 
\begin{theorem}[\cite{Jozefiak1979,Jozefiak1981,Harris1984a}]\label{prop:ad-chow}
Let $\X$ be a connected scheme which is smooth, (or, more generally, Cohen-Macaulay).  Let $\sE$ be a vector bundle over $\X$ of rank $r$, and $\rho \in \cohlgy[0]{\X,\ext[2]{\sE}}$ a skew form.  For $k \ge 0$ an integer, let $d(r,k) = {r-2k \choose 2}$, and for $0\le j \le n$, denote by $c_j = c_j(\sE) \in \chow[j]{\X}$ the Chern classes of $\sE$ in the Chow ring of $\X$.  If $\D[2k]{\rho}$ is non-empty, then each of its components has codimension $\le d(r,k)$ in $\X$.  Moreover, the class
$$
\det\rbrac{
\begin{matrix}
c_{n-k} 	& c_{n-k+1}	& \cdots &\\
c_{n-k-2} 	& c_{n-k-1} & 		 &\\
\vdots		&			& \ddots & \\
			&			&		& c_1
\end{matrix}}
\in \chow[d(r,k)]{\X}
$$
is supported on $\D[2k]{\rho}$ and is equal to the fundamental class of $\D[2k]{\rho}$ if the dimensions agree.
\end{theorem}

Bondal's conjecture suggests that for the Poisson structure $\ps$, the codimensions predicted by \autoref{prop:ad-chow} are likely to be much higher than the actual codimension of $\D[2k]{\ps}$.  We therefore expect similar behaviour from the degeneracy loci $\AD[2k]{\nabla}$ of a module.  However, there is one particular case in which this estimate of the codimension will be useful to us:
\begin{corollary}\label{cor:adapt-codim3}
Let $(\X,\ps)$ be a smooth connected Poisson scheme of even dimension $2n$ and for $1 \le j \le 2n$, let $c_j = c_j(\tshf{X}) \in \chow[j]{\X}$ be its $j^{th}$ Chern class.  If $(\sL,\nabla)$ is an invertible Poisson module and $\AD[2n-2]{\nabla}$ is non-empty, then every component of $\AD[2n-2]{\nabla}$ has codimension $\le 3$ in $\X$.

If the codimension of every component is exactly equal to three, then
$$
[\AD[2n-2]{\nabla}]=c_1c_2-c_3,
$$
and there is a locally-free resolution
$$
\xymatrix{
0 \ar[r] &  \can_\X^{2} \ar[r]^-{\ps_\nabla^n} & \sA_\sL^\vee \otimes \can_\X \ar[r]^-{\pss_\nabla\otimes 1 } & \sA_\sL\otimes \can_\X \ar[r]^-{\ps_\nabla^n} & \strc{X} \ar[r] & \strc{\AD[\mathrm{2}\mathit{n}-\mathrm{2}]{\nabla}} \ar[r] & 0
}
$$
where
$$
\ps^n_\nabla \in {\sA^{2n}_\sL} \cong \sA_\sL^\vee \otimes \can_\X
$$
is the $n^{th}$ exterior power.  In particular, $\AD[2n-2]{\nabla}$ is Gorenstein with dualizing sheaf
$$
\can_{\AD[2n-2]{\nabla}} \cong \can_\X^{-1}|_{\AD[2n-2]{\nabla}}.
$$
\end{corollary}

\begin{proof}
In this case, the vector bundle is the Atiyah algebroid $\sA_\sL$, which has rank odd rank $2n+1$.  Since $\sA_\sL$ is an extension
$$
\xymatrix{
0 \ar[r] & \strc{X} \ar[r] & \sA_\sL \ar[r] & \der[1]{\X} \to 0,
}
$$
we have $c_j(\sA_\sL) = c_j(\X)$ for all $j$.  The statements about the codimension and fundamental class are therefore contained in \autoref{prop:ad-chow}.

Locally, we are interested in the vanishing of the submaximal Pfaffians of a $(2n+1)\times(2n+1)$ matrix.  The fact that such degeneracy loci are Gorenstein when the codimension is three, and that the free resolutions have the form in question, was proven by Buchsbaum and Eisenbud~\cite{Buchsbaum1977}.  The global version of the resolution for skew-symmetric bundle maps is described by Okonek in \cite{Okonek1994}; we simply apply the formula therein, noting that for the Atiyah algebroid $\sA_\sL$, we have $\det \sA_\sL \cong \can_\X^{-1}$.
\end{proof}

\begin{remark}
We originally found this three-term resolution in the case when $\sL = \can_\X^{-1}$ by using the geometry of logarithmic vector fields.  We are grateful to Ragnar-Olaf Buchweitz for suggesting that these subschemes might be Gorenstein and referring us to the work of Buchsbaum and Eisenbud, which allowed us to generalize and simplify the result.\qed
\end{remark}

\section{Characteristic classes of Poisson modules}
\label{sec:char-class}
In this section we discuss the Chern classes of Poisson modules.  In particular, we give a vanishing theorem for adapted modules, which was our original motivation for the definition of adaptedness.  Using this result, we obtain lower bounds on the dimension of the degeneracy locus of the Poisson structure.  We will need to use some basic results regarding the Borel-Moore homology of complex varieties, which we presently recall. We refer the reader to \cite[Chapter 19]{Fulton1998} for more details.

If $\X$ is a scheme of finite type over $\complex$, we denote by $\X^{an}$ the associated complex analytic space, and by $\cohlgy{\X} = \cohlgy{\X^{an},\complex}$ its singular cohomology.  A locally-free sheaf $\sE$ on $\X$ then has Chern classes
$$
\chern[p]{\sE} \in \cohlgy[2p]{\X}.
$$
We denote by $\Chern{\sE}$ the subring generated by the Chern classes.

We denote by $\hlgy{\X}$ the Borel-Moore homology groups of $\X^{an}$ with complex coefficients.  These groups satisfy the following properties:
\begin{enumerate}
\item If $\Y \subset \X$ is a closed subscheme with complement $\U = \X \setminus \Y$, there is a long exact sequence
\begin{equation}\label{eqn:les-bm-hlgy}
\xymatrix{
\cdots \ar[r] & \hlgy[j]{\Y} \ar[r] & \hlgy[j]{\X} \ar[r] & \hlgy[j]{\U} \ar[r] & \hlgy[j-1]{\Y} \ar[r] & \cdots 
}
\end{equation}
\item If $\dim \X = n$, then $\hlgy[j]{\X} = 0$ for $j > 2n$, and $\hlgy[2n]{\X}$ is the vector space freely generated by the $n$-dimensional irreducible components of $\X$.  Moreover, there is a fundamental class $[\X] \in \hlgy[2n]{\X}$ which is a linear combination of these generators, with coefficients given by the multiplicities of the components.
\item There are cap products $\hlgy[j]{\X} \otimes \cohlgy[k]{\X} \to \hlgy[j-k]{\X}$ satisfying the usual compatibilities.
\end{enumerate}

\begin{theorem}\label{thm:bottvan}
Let $(\X,\ps)$ be a regular Poisson scheme, so that $\ps$ has constant rank $2k$.  If $\sE$ is a locally-free, adapted Poisson module, then its Chern ring vanishes in degree $> 2(n-2k)$:
$$
\Chern[p]{\sE} = 0 \subset \cohlgy[p]{\X}
$$
if $p > 2(n-2k)$.
\end{theorem}

\begin{proof}
Since $\X$ is smooth and $\ps$ has constant rank, the image of $\ps$ defines an involutive subbundle $\sF{} \subset \tshf{X}$.  Since $\sE$ is an adapted Poisson module, it admits a flat partial connection $\nabla : \sE \to \sFd{}\otimes \sE$ along $\sF$.  Therefore the proof of Bott's vanishing theorem \cite{Bott1972} applies (substituting $\sE$ for the normal bundle $\tshf{X}/\sF{}$) and establishes the theorem.
\end{proof}
%\begin{remark}\rm
%One may also consider Chern classes in the sheaf cohomology group $\cohlgy[p]{\X,\forms[p]{\X}}$, or the algebraic de Rham cohomology $\cohlgy[2p]{\X}[dR]$.  See, eg. \cite{Illusie1971}.
%\end{remark}

Baum and Bott~\cite{Baum1972} applied Bott's vanishing theorem to a singular foliation, obtaining ``residues'' in the homology of the singular set which are Poincar\'e dual to the Chern classes of the normal sheaf of the foliation.  In the same way, we apply \autoref{thm:bottvan} to a locus $\Y$ on which a Poisson module $\sL$ is generically adapted, obtaining residues in the homology of the Poisson singular locus of $\Y$ which are dual to the Chern classes of $\sL$:
\begin{corollary}\label{cor:adapt-residue}
Let $(\Y,\ps)$ be an irreducible Poisson scheme of dimension $n$, and let $(\sL,\nabla)$ be an invertible Poisson module.  Suppose that the rank of $\ps$ is generically equal to $2k$, and that $\AD[2k]{\nabla}=\Y$, so that the module is adapted away from the Poisson singular locus $\Z = \Sing{\Y,\ps}$.

If $p > n-2k$, then there is a class $R_p \in \hlgy[2n-2p]{\Z}$ such that
$$
i_*(R_p) = \chern[1]{\sL}^p \cap [\Y] \in \hlgy[2n-2p]{\Y},
$$
where $i : \Z \to \Y$ is the inclusion.  In particular, if
$$
\chern[1]{\sL}^{n-2k+1} \cap [\Y] \ne 0
$$
then $\Z$ has a component of dimension $\ge 2k+1$.
\end{corollary}

\begin{proof}
Let $\Y^\circ = \Y \setminus \Z$ be the regular locus.  Since $\nabla|_{\Y^\circ}$ is adapted, the previous theorem informs us that $\chern[1]{\sL|_{\Y^\circ}}^p = 0$, and hence
$$
\chern[1]{\sL|_{\Y^\circ}}^p \cap [\Y^\circ] = 0 \in \hlgy[2n-2p]{\Y^\circ}.
$$
Appealing to the long exact sequence \eqref{eqn:les-bm-hlgy}, we find the desired class $R_p$.  If
$$
\chern[1]{\sL}^p\cap[\Y] \ne 0,
$$
it follows that $R_p \ne 0$, and so the statement regarding the dimension of $\Z$ now follows from the vanishing $\hlgy[j]{\Z} = 0$ for $j > 2\dim \Z$.
\end{proof}

\begin{remark}
Under certain non-degeneracy conditions, Baum and Bott relate their homology classes to the local behaviour of the foliation near the singular set of the foliation.  One therefore wonders whether there may be a connection between the residue multiderivations defined in \autoref{sec:defn-res} and the homology classes in the Poisson singular locus.  We hope to address this issue in future work.\qed
 \end{remark}

%\begin{remark}
%One can make similar statements for modules of higher rank, but one must be careful about when the module actually restricts to an adapted module on the appropriate loci.  We will not need such generalizations in this paper.
%\end{remark}

We are particularly interested in applying \autoref{cor:adapt-residue} to an ample Poisson module, because the positivity of its Chern class implies the nonvanishing of the homology classes $R_p$ described above, placing lower bounds on the dimensions of certain degeneracy loci.

\begin{lemma}
Let $(\X,\ps)$ be a Poisson scheme with $\X$ a smooth projective variety.  Let $(\sL, \nabla)$ be a Poisson module, and suppose that $\sL$ is an ample invertible sheaf.  Let $k>0$, and let $\Y$ be an  irreducible, closed, strong Poisson subscheme of $\AD[2k]{\nabla}$ which is not contained in $\D[2k-2]{\ps}$.  Then the Poisson singular locus $\Sing{\Y,\ps|_\Y}$ is non-empty and has a component of dimension $\ge 2k-1$.
\end{lemma}

\begin{proof}
Since $\Y$ is a strong Poisson subscheme of $\AD[2k]{\nabla}$, the module $\nabla$ restricts to a Poisson module on $\Y$, and we have $\AD[2k]{\nabla|_\Y} = \Y$.  Hence the pair $(\Y,\nabla|_\Y)$ satisfies the hypotheses of \autoref{cor:adapt-residue}.

Let $d$ be the dimension of $\Y$.  Then $d \ge 2k$ because $\ps|_\Y$ has generic rank $2k$.  Since $k >0$, we have
$$
0 < d-2k+1 < d.
$$
Since $\sL$ is ample, so is $\sL|_{\Y}$, and hence
$$
\chern[1]{\sL|_\Y}^{d-2k+1} \cap [\Y] \ne 0 \in \hlgy[2(2k-1)]{\Y}
$$
Therefore
$$
\dim \Sing{\Y,\ps|_\Y} \ge 2k-1,
$$
as required.
\end{proof}

\begin{corollary}\label{cor:ample-degen}
Let $(\X,\ps)$ be a Poisson scheme with $\X$ smooth and projective, and suppose that $(\sL,\nabla)$ is an ample invertible Poisson module.  If $\AD[2k]{\nabla}$ is non-empty and has a component of dimension $\ge 2k-1$, then $\D[2k-2]{\ps}$ is also non-empty and has a component of dimension $\ge 2k-1$.
\end{corollary}

\begin{proof}
Select an irreducible component $\Y_0'$ of $\AD[2k]{\nabla}$ of maximal dimension, and let $\Y_0 = (\Y_0')_{red}$ be the reduced subscheme.  For $j > 0$, select by induction an irreducible component $\Y_j'$ of $\Sing{\Y_{j-1},\ps|_{\Y_{j-1}}}$ of maximal dimension, and let $\Y_j = (\Y_j')_{red}$.  We therefore obtain a decreasing chain
$$
\Y_0 \supset \Y_1 \supset \Y_2 \supset \cdots
$$
of strong Poisson subschemes of $\AD[2k]{\nabla}$, with each inclusion strict.

Since $\X$ is Noetherian, the chain must eventually terminate, and so there is a maximal integer $J$ such that $\dim \Y_{J} \ge 2k-1$.  We claim that the rank of $\ps|_{\Y_{J}}$ is $\le 2k-2$.  Indeed, if it were not, then $\Y_{J}$ would satisfy the conditions of the previous lemma, and so we would find that $\Y_{J+1}$ is non-empty of dimension $\ge 2k-1$, contradicting the maximality of $J$.
\end{proof}

\section{Degeneracy loci and the canonical module}
%\section{The structure of the degeneracy loci}
\label{sec:degen}

\subsection{The modular residues}\label{sec:mod-res}

In this section we will give a simple formula for the residues of the canonical module, and relate it to the geometry of the degeneracy loci.  Throughout this section, $(\X,\ps)$ will be a smooth connected Poisson scheme.

For a multiderivation $U \in \der{\X}$, we denote by $\iota_U : \forms{\X} \to \forms{\X}$ the operator of interior contraction by $U$.  Then the Schouten bracket $[U,V]$ of sections $U, V\in \der{\X}$ is characterized by the identity
\begin{align}
\iota_{[U,V]} = [[\iota_{U},d],\iota_{V}],\label{eqn:schout-id}
\end{align}
where the brackets on the right-hand side are graded commutators.

\begin{lemma}\label{lem:ps-comm}
For a Poisson structure $\ps \in \der[2]{\X}$, we have the commutator
$$
[\ips[k],d] = k \ips[k-1] [\ips,d]
$$
of operators on $\forms{\X}$.
\end{lemma}
\begin{proof}
Let $U = [\ps,d]$.  Since $[\ps,\ps] = 0$, the formula \eqref{eqn:schout-id} gives the commutator
$$
[U,\ips] = 0.
$$
For $k=1$ the formula we seek is trivially satisfied, so suppose by induction that it is correct for some $k \ge 1$.  The derivation property of the commutator gives
$$
[\ips[k+1],d] = [\ips[k],d]\ips + \ips[k][\ips,d]
$$
which, by the inductive hypothesis, gives
\begin{align*}
[\ips[k+1],d] &= k\ips[k-1]U\ips +  \ips[k]U \\
&= (k+1)\ips[k]U,
\end{align*}
since $U$ and $\ips$ commute.
\end{proof}

Recall from \autoref{ex:can-mod} that the Poisson module structure on $\can_\X$ is defined by the formula
$$
\nabla_\alpha \mu =  -\alpha \wedge d \ips \mu
$$
for $\mu\in\can_\X$ and $\alpha \in \forms[1]{\X}$.  For $\mu$ a local trivialization with corresponding connection vector field $Z$, we have
\begin{align*}
\nabla_\alpha \mu &= (\iota_Z \alpha)\mu \\
&= \iota_Z(\alpha \wedge \mu) + \alpha \wedge \iota_Z \mu \\
&= \alpha \wedge \iota_Z \mu,
\end{align*}
because degree considerations force $\alpha \wedge \mu = 0$.  Therefore, $Z$ is uniquely determined by the formula
\begin{align*}
\iota_Z \mu &= -d\ips\mu  \\
&= [\ips,d]\mu 
\end{align*}
since $d\mu = 0$.  Using this formula we can compute the residues:
\begin{theorem}\label{thm:mod-res-formula}
For $k \ge 0$, denote by
$$
D\ps^{k+1} = j^1(\ps^{k+1}) | _{\D[2k]{\ps}} \in \forms[1]{\X} \otimes \der[2k+1]{\X}|_{\D[2k]{\ps}}
$$
the derivative of $\ps^{k+1}$ along its zero set, and denote by
$$
\tr : \forms[1]{\X} \otimes \der[2k+2]{\X} \to \der[2k+1]{\X}
$$
the contraction.  Then the $k^{th}$ modular residue is given by the formula
$$
i_*\mres{k}{\ps} = \frac{-1}{k+1}\tr(D\ps^{k+1}) \in \der[2k+1]{\X}|_{\D[2k]{\ps}},
$$
where $i : \D[2k]{\ps} \to \X$ is the inclusion.
\end{theorem}
\begin{proof}
Let $\Y = \D[2k]{\ps}$.  The question is local, so we may pick a trivialization $\mu \in \can_\X$.  Let $Z$ be the connection vector field with respect to this trivialization, so that $\iota_Z \mu = [\ips,d]\mu$.  

By \autoref{rem:higgs-res} it suffices to show that
$$
(Z \wedge \ps^k) |_\Y = \frac{-1}{k+1}\tr(D\ps^{k+1}).
$$ 
We compute
\begin{align*}
\iota_{Z\wedge\ps^k} \mu &= \ips[k]\iota_Z \mu \\
&= \ips[k][\ips,d]\mu \\
&= \frac{1}{k+1} [\ips[k+1],d] \mu
\end{align*}
by \autoref{lem:ps-comm}.  Therefore
$$
\iota_{Z \wedge \ps^k}\mu = \frac{-1}{k+1} d\ips[k+1]\mu,
$$
since $\mu$ is closed.  But $\ips[k+1]\mu$ vanishes on $\D[2k]{\ps}$, and hence its one-jet restricts to the derivative
$$
D(\ips[k+1]\mu) \in \forms[1]{\X} \otimes \forms[n-2k-2]{\X} |_\Y,
$$
where $n = \dim \X$.  Since the symbol of the exterior derivative is the exterior product $\epsilon$, we have that
$$
\iota_{Z \wedge \ps^k}\mu|_\Y = \frac{-1}{k+1}\epsilon(D(\ips[k+1]\mu)) \in \forms[n-2k-1]{\X}|_\Y.
$$
Now, the Hodge isomorphism $\der{\X} \to \forms[n-\bullet]{\X}$ defined by $\mu$ intertwines interior contraction and exterior product, so this formula shows that
$$
(Z \wedge \ps^k)|_\Y = \frac{-1}{k+1} \tr(D\ps^{k+1}),
$$
as desired.
\end{proof}

\begin{remark}
Suppose that $\X$ is smooth of dimension $2n$, and $\ps$ is generically symplectic.  Suppose further that the degeneracy divisor $\Y = \D[2n-2]{\ps}$ is also smooth.  Inverting $\ps^n$ gives a meromorphic volume form $\mu = (\ps^n)^{-1}$ with poles on $\Y$, which has a Poincar\'e residue $\res{}{\mu} \in \cohlgy[0]{\Y,\omega_\Y}$.  We have the identity
$$
\mres{n-1}{\ps} = -n\res{}{\mu}^{-1} \in \cohlgy[0]{\Y,\can_\Y^{-1}},
$$
so that the modular residues of top degree are residues in the usual sense.\qed
\end{remark}

\begin{example}
For the Poisson structure
$$
\ps = (x\cvf{x}+y\cvf{y})\wedge \cvf{z}
$$
on $\Aff{3} = \Spec(\complex[x,y,z])$ considered in \autoref{ex:euler-planes}, the zero locus is the $z$-axis $\Z \subset \Aff{3}$.  To compute the modular residue, we calculate the derivative
$$
D\ps =  \rbrac{dx \otimes (\cvf{x}\wedge\cvf{z}) + dy \otimes (\cvf{y}\wedge\cvf{z}) } |_\Z,
$$
and contract the one-forms into the bivectors to find
$$
\mres{0}{\ps} = - \tr(D\ps) = -2 \cvf{z}|_\Z,
$$
giving a trivialization of $\der[1]{\Z}$.\qed
\end{example}

While the degeneracy loci of a Poisson module are always Poisson subschemes, the degeneracy loci of the canonical module are, in fact, strong Poisson subschemes:
\begin{proposition}
Let $(\X,\ps)$ be a smooth Poisson scheme.  Then the degeneracy locus $\AD[2k][mod]{\ps}$ of the canonical module is a strong Poisson subscheme.
\end{proposition}

\begin{proof}
Again, the question is local, so we may pick a trivialization $\phi \in \can_\X$ with connection vector field $Z$.  We are interested in the simultaneous vanishing set of the tensors $\ps^{k+1}$ and $Z \wedge \ps^k$.  If $U \in \der[1]{\X}$ is a Poisson vector field, we compute
\begin{align*}
-(k+1) \lie{U}(\iota_{Z\wedge \ps^k} \mu) &= \lie{U}(d\ips[k+1]\mu) \\
&=  d \lie{U}(\ips[k+1]\mu) \\
&=  d(\ips[k+1] \lie{U}\mu)
\end{align*}
since $\lie{U}\ps = 0$.  Now $\lie{U}\mu = f \mu$ for some $f \in \strc{X}$, and hence
\begin{align*}
-(k+1)\lie{U}(\iota_{Z\wedge \ps^k} \mu) &=  d(\ips[k+1]f \mu) \\
&=  df \wedge \ips[k+1]\mu + f d \ips[k+1]\mu \\
&= df  \wedge \ips[k+1]\mu -(k+1) f\iota_{Z \wedge \ps^k} \mu
\end{align*}
which vanishes on $\AD[2k][mod]{\ps}$.  Hence the ideal defining the $2k^{th}$ degeneracy locus is preserved by $U$ as required.
\end{proof}

\subsection{Singularities of degeneracy loci}

We are now in a position to study the structure of the degeneracy loci in detail.  We begin with the following observation:
\begin{lemma}
Let $(\X,\ps)$ be a smooth Poisson scheme, and let $k \ge 0$.  Then the one-jet $j^1\ps^{k+1}$ vanishes on $\D[2k-2]{\ps}$.
\end{lemma}

\begin{proof}
The question is local, so we may choose a connection $\nabla: \der[1]{\X} \to \forms[1]{\X} \otimes \der[1]{\X}$. which induces connections on $\der[2j]{\X}$ for all $j$.  Using the connection, we may write
$$
j^1 \ps^{k+1} = (\ps^{k+1},(k+1)\nabla\ps \wedge \ps^k) \in \der[2k+2]{\X} \oplus \rbrac{\forms[1]{\X}\otimes\der[2k-2]{\X}}
$$
by the Leibniz rule.  Hence $j^1\ps^{k+1}$ vanishes when $\ps^k$ does.
\end{proof}

\begin{corollary}\label{thm:rank-sing}
Suppose that $(\X,\ps)$ is a smooth connected Poisson scheme and $\D[2k]{\ps} \ne \X$.  Then
$$
\D[2k-2]{\ps} \subset \D[2k]{\ps}_{sing},
$$
so that the Poisson structure has rank $2k$ on the smooth locus of $\D[2k]{\ps}$.  In particular, every reduced component of $\D[2k]{\ps}$ has dimension $\ge 2k$.
\end{corollary}

\begin{proof}
If $x \in \D[2k]{\ps}$ is a smooth point, then $j^1(\ps^{k+1})_x \ne 0$.  By the lemma, we have $\ps^k_x \ne 0$ and hence the rank of $\ps$ at $x$ is $\ge 2k$.  Therefore, $\D[2k-2]{\ps}$ is contained in the singular locus.  If $\Y$ is a component of $\D[2k]{\ps}$ which is reduced, then it has a smooth point, and so the rank of $\ps|_\Y$ is generically $2k$.  But $\Y$ is a Poisson subscheme, and hence it must have dimension $\ge 2k$.
\end{proof}

\begin{remark}
%In light of this result, it is natural to wonder if every component of $\D[2k]{\ps}$ has dimension $\ge 2k$.  It is certainly not the case in the $C^\infty$ category: there are Poisson structures on $\real^{2n}$ which are symplectic except at the origin, where they vanish~\cite{Zakrzewski2000}.
In light of this result, it is natural to wonder if every component of $\D[2k]{\ps}$ has dimension $\ge 2k$ whenever it is nonempty.  However, it is not the case.  Recall that if $\g$ is a Lie algebra, then the Lie bracket induces the Kirillov-Kostant-Souriau Poisson structure $\ps$ on $\g^\vee$, whose coefficients are linear polynomials on $\g^\vee$.  The symplectic leaves of $\ps$ are exactly the coadjoint orbits.  But the minimal dimension of a non-zero coadjoint orbit is rarely equal to two~\cite{Wolf1978}.  For example, every non-zero coadjoint orbit of $\mathfrak{sl}(3,\complex)$ has dimension at least four~\cite{Graham2005a}.  Hence, for many Lie algebras we have $\D[2]{\ps} = \{0\}$ as sets.  However, the previous result shows that the scheme structure cannot be reduced.  But this is clear: the ideal of $\D[2]{\ps}$ it is generated by the coefficients of $\ps\wedge \ps$, which are homogeneous quadratic polynomials on $\g^\vee$, and hence they do not generate the maximal ideal of the origin.\qed
\end{remark}

In the case of a generically symplectic Poisson structure, we can use the canonical module to study the structure of the singular locus of the degeneracy divisor:
\begin{theorem}\label{thm:adapt-sing}
Let $(\X,\ps)$ be a smooth connected Poisson scheme of dimension $2n$, and suppose that $\ps$ is generically symplectic.  Let $\Y = \D[2n-2]{\ps}$ be the anti-canonical divisor on which $\ps$ degenerates.  Then 
$\Y_{sing} = \AD[2n-2][mod]{\ps}$
%$$\Y_{sing} = \D[2n-2]{\sigma_{\nabla^{\omega_X}}} = \D[2n-2]{\sigma_{{mod}}}$$
 as schemes.  In particular, every component of $\Y_{sing}$ has dimension $\ge 2n-3$ if it is non-empty. Moreover, if $\dim \Y_{sing} = 2n-3$, then $\Y_{sing}$ is Gorenstein with dualizing sheaf $\can_\X^{-1}|_{\Y_{sing}}$ and its fundamental class is
$$
[\Y_{sing}] = c_1c_2-c_3 \in \chow[3]{\X}.
$$
\end{theorem}

\begin{proof}
According to \autoref{prop:adapt-locus}, the subscheme $\AD[2n-2][mod]{\ps}$ is defined by the vanishing of the section
$$
i_*\mres{n-1}{\ps} \in \der[2n-1]{\X}|_\Y
$$
where $i : \Y \to \X$ is the inclusion.  By \autoref{thm:mod-res-formula} we have
$$
i_*\mres{n-1}{\ps} = -\frac{1}{n}\tr(D\ps^n) \in \cohlgy[0]{\Y,\der[2n-1]{\X}|_\Y}.
$$
Notice that the contraction
$$
\tr : \forms[1]{\X} \otimes \der[2n]{\X} \to \der[2n-1]{\X}
$$
is an isomorphism, so the vanishing scheme of $\tr(D\ps^n)$ coincides identically with the vanishing scheme of $D\ps^n$ itself.  But the latter section defines the singular scheme of $\Y$.  Hence $\Y_{sing} = \AD[2n-2][mod]{\ps}$ as schemes.  The rest of the theorem now follows from \autoref{cor:adapt-codim3}.
\end{proof}

\subsection{Fano manifolds}
We now describe several properties of the degeneracy loci of Poisson structures on Fano varieties. 

\begin{lemma}\label{lem:fano-forms}
If $\X$ is a Fano manifold then $\cohlgy[0]{\X,\forms[q]{\X}} = 0$ for $q > 0$.
\end{lemma}

\begin{proof}
By the Hodge decomposition theorem, the space in question is the complex conjugate of $\cohlgy[q]{\X,\strc{\X}}$, which is Serre dual to $\cohlgy[n-q]{\X,\omega_\X}$.  But $\omega_\X$ is anti-ample, so the latter space is zero by Kodaira's vanishing theorem.
\end{proof}

\begin{corollary}\label{cor:no-der-CY}
If $\X$ is a Fano manifold of dimension $n>2$ and $\Y \subset \X$ is a smooth anti-canonical divisor, then
$$
\cohlgy[0]{\Y,\der[q]{\Y}} = 0
$$
for $0<q<n-1$.
\end{corollary}

\begin{proof}
The line bundle $\can_\Y$ is trivial by adjunction.  We therefore have
$$
\der[q]{\Y} \cong \forms[n-1-q]{\Y}.
$$
But if $0<q<n-1$, then $0<n-1-q < n-1$, and hence
$$
\cohlgy[0]{\Y,\forms[n-1-q]{\Y}}\cong \cohlgy[0]{\X,\forms[n-1-q]{\X}} = 0
$$
by the Lefschetz hyperplane theorem and \autoref{lem:fano-forms}.
\end{proof}

Polishchuk \cite[Theorem 11.1]{Polishchuk1997} showed that if $\ps$ is a Poisson structure on the projective space $\Prj{d}$, and if $\Y \subset \Prj{d}$ is an anti-canonical divisor which is a Poisson subscheme, then $\Y$ must be singular.  This result implies, in particular, that the divisor on which a generically symplectic Poisson structure on $\Prj{2n}$ degenerates is always singular, although it places no bounds on the dimension of the singular locus.  We can now considerably strengthen this statement about the degeneracy divisor:
\begin{theorem}\label{thm:degen-sing}
Let $(\X,\ps)$ be a Poisson Fano manifold of dimension $2n$ with $n > 1$, and suppose that $\ps$ is generically symplectic.  Then the degeneracy divisor $\D[2n-2]{\ps} \subset \X$ is singular, and every component of its singular subscheme has dimension $\ge 2n-3$.
\end{theorem}

\begin{proof}
If $\Y=\D[2n-2]{\ps}$ were smooth, the Poisson structure would vanish on $\Y$ by \autoref{cor:no-der-CY}.  But then $\Y$ would be singular by \autoref{thm:rank-sing}, a contradiction.  The statement about the dimension now follows from \autoref{thm:adapt-sing}.
\end{proof}

Now, using the non-emptiness of the singular locus, we deduce the non-emptiness of the degeneracy locus:
\begin{theorem}\label{thm:bondal}
\mainthm
\end{theorem}

\begin{proof}
If the rank of $\ps$ is less than $2n$ everywhere, then $\D[2n-2]{\ps} = \X$, and $\D[2n-4]{\ps}$ has dimension $\ge 2n-3$ by \autoref{thm:polish}.  So, we may assume that $\ps$ is generically symplectic, ie. that the section $\ps^n \in \cohlgy[0]{\X,\ext[2n]{\tshf{X}}}$ is non-zero.

Let $\Y = \D[2n-2]{\sigma}$ be the zero locus of $\ps^n$, which is non-empty of dimension $2n-1$ because $\X$ is Fano.  By the previous theorem, $\Y$ is singular, and hence \autoref{thm:adapt-sing} informs us that $\AD[2n-2][mod]{\ps}$ is non-empty of dimension $\ge 2n-3$.  Now we are in the situation of \autoref{cor:ample-degen}, and we conclude that $\D[2n-4]{\ps}$ is non-empty of dimension $\ge 2n-3$, as desired.
\end{proof}

\begin{corollary}
Bondal's \autoref{con:bondal} holds for Fano manifolds of dimension four.
\end{corollary}

\section{Residues and elliptic normal curves}
\label{sec:elliptic}
In this section, we apply our results to Poisson structures associated with elliptic normal curves.  Let $\X$ be an elliptic curve, and let $\sE$ be a stable vector bundle over $\X$ of rank $r$ and degree $d$, with $\gcd(r,d)=1$.  Feigin and Odesski{\u\i}~\cite{Feigin1989,Feigin1998} have explained that the projective space $\Prj{d-1} = \Prj{}(\Ext[1]{\sE,\strc{X}}[\X])$ inherits a natural Poisson structure; see also \cite{Polishchuk1998} for a generalization of their construction.

For the particular case when $\sE = \sL$ is a line bundle of degree $d$, the projective space is
$$
\Prj{d-1} = \Prj{}(\cohlgy[1]{\X,\sL^\vee}),
$$
which is the Kodaira embedding space for $\sL$.  Hence $\X$ is embedded in $\Prj{d-1}$ as an elliptic normal curve of degree $d$.  

The vector space $\cohlgy[1]{\X,\sL^\vee}$ inherits a homogeneous Poisson structure~\cite{Bondal1993,Feigin1989}.  This homogeneous Poisson structure has a unique homogeneous symplectic leaf of dimension $2k+2$, given by the cone over the $k$-secant\footnote{We warn the reader that our index $k$ differs from other sources.} variety $\Sec[k]{\X}$~\cite{Feigin1989}---that is, the union of all the $k$-planes in $\Prj{d-1}$ which pass through $k+1$ points on $\X$ (counted with multiplicity).   It follows that the induced Poisson structure $\ps$ on $\Prj{d-1}$ has rank $2k$ on $\Sec[k]{\X}\setminus\Sec[k-1]{\X}$, and so
$$
\D[2k]{\ps} \cap \Sec[k+1]{\X} = \Sec[k]{\X}
$$
for all $k < (d-1)/2$.  When $d=2n+1$ is odd, this gives an equality
$$
\Sec[k]{\X} = \D[2k]{\ps}_{red}
$$
of the underyling reduced schemes, because $\ps$ is symplectic on $\Prj{2n} \setminus \Sec[n-1]{\X}$.  When $d$ is even, though, $\D[2k]{\ps}$ may have additional components.  For example, when $d=4$, the Poisson structure on $\Prj{3}$ vanishes on the elliptic curve $\X \subset \Prj{3}$ together with four isolated points~\cite{Polishchuk1997}.  We believe that $\Sec[k]{\X}$ is always a component of $\D[2k]{\ps}$ and that the latter scheme is reduced; see \autoref{con:reduced}.

The secant varieties have been well-studied:
\begin{theorem}[{\cite[\S8]{GrafvBothmer2004}}]\label{thm:secants}
The secant variety $\Sec[k]{\X}$ enjoys the following properties:
\begin{enumerate}
\item $\Sec[k]{\X}$ has dimension
$$
\dim{\Sec[k]{\X}} = 2k+1
$$
and degree
$$
\deg \Sec[k]{\X} = { d - k -2 \choose k} + { d-k-1 \choose k+1}.
$$
\item $\Sec[k]{\X}$ is normal and arithmetically Gorenstein.  Its dualizing sheaf is trivial.
\item $\Sec[k]{\X} \setminus \Sec[k-1]{\X}$ is smooth.
\end{enumerate}
\end{theorem}

\begin{remark}
Since $\Sec[k]{\X}$ is contained in $\D[2k]{\ps}$, we have that $$\dim{\D[2k]{\ps}} \ge 2k+1$$ for $2k < d-1$ in accordance with \autoref{con:bondal}.  Indeed, Bondal cited these examples as motivation for his conjecture.\qed
\end{remark}

With this information in hand, we can now make a conclusion about the singular scheme of the highest secant variety of an odd-degree curve:
\begin{theorem}\label{thm:sec-sing}
Let $\X \subset{\Prj{2n}}$ be an elliptic normal curve of degree $2n+1$, and let $\Y = \Sec[n-1]{\X}$, a hypersurface of degree $2n+1$.  Then the singular subscheme $\Y_{sing}$ of $\Y$ is a $(2n-3)$-dimensional Gorenstein scheme of degree $2{2n+2 \choose 3}$ whose dualizing sheaf is $\strc{\Y_\mathit{sing}}(2n+1)$.
\end{theorem}

\begin{proof}
By the discussion above, $\Y$ is the reduced scheme underlying the degeneracy locus $\D[2n-2]{\ps}$, which is an anti-canonical divisor.  Since the degree of $\can_{\Prj{2n}}^{-1}$ is $2n+1$, we must have $\Y = \D[2n-2]{\ps}$ as schemes.  The singular subscheme is supported on the subvariety $\Sec[n-2]{\X}$, which has codimension three in $\Prj{2n}$.  By \autoref{thm:adapt-sing}, $\Y_{sing}$ is Gorenstein with dualizing sheaf given by the restriction of $\can_{\Prj{2n}}^{-1} \cong \strc{\Prj{\mathit{n}}}(2n+1)$.  Its fundamental class is given by $c_1c_2-c_3$, where $c_j$ is the $j^{th}$ Chern class of $\Prj{2n}$.  But $c_j$ has degree $2n+1 \choose j$, and hence one obtains the formula for the degree of $\Y_{sing}$ by a straightforward computation.
\end{proof}

\begin{corollary}
In the situation of the previous theorem, $\Y_{sing}$ is not reduced.  Its geometric multiplicity (that is, the length of its structure sheaf at the generic point) is equal to eight.
\end{corollary}

\begin{proof}
The reduced scheme underlying $\Y_{sing}$ is $\Sec[n-2]{\X}$.  Using the formula in~\autoref{thm:secants}, we find that the degree of $\Y_{sing}$ is eight times that of $\Sec[n-2]{\X}$, and the result follows.
\end{proof}

We now return to the case when $d$ is arbitrary and study the modular residues:
\begin{proposition}\label{prop:res-nonvan}
The modular residue
$$
\mres{k}{\ps} \in \cohlgy[0]{\D[2k]{\ps},\der[2k+1]{{\D[2k]{\ps}}} }
$$
induces a non-vanishing multiderivation on $\Sec[k]{\X}\subset\D[2k]{\ps}$, in accordance with the fact that the dualizing sheaf of $\Sec[k]{\X}$ is trivial.
\end{proposition}

\begin{proof}
Because $\can_{\Prj{d-1}} \cong \strc{\Prj{\mathit{d}-1}}(-d)$, the tautological bundle is a Poisson module.  With respect to this Poisson module structure, the residues of $\strc{\Prj{\mathit{d}-1}}(-1)$ are non-zero multiples of the residues of $\can_{\Prj{d-1}}$.  We may therefore work with $\strc{\Prj{\mathit{d}-1}}(-1)$ instead of $\can_{\Prj{d-1}}$.

Let $\V = \cohlgy[1]{\X,\sL^\vee}$ so that $\Prj{d-1} = \Prj{}(\V)$.   Let $\V'$ be the total space of the tautological bundle $\strc{\Prj{\mathit{d}-1}}(-1)$ and let $\pi : \V' \to \Prj{d-1}$ be the projection.  Being the total space of an invertible Poisson module, $\V'$ is a Poisson variety.  For a subscheme $\Y \subset \Prj{d-1}$, the preimage $\pi^{-1}(\Y)$ in $\V'$ is a Poisson subscheme if and only if $\Y$ is a Poisson subscheme of $\Prj{d-1}$ and $\nabla|_\Y$ is a Poisson module.  (This fact follows from the definition of the Poisson bracket on $\V'$ in \autoref{sec:mod-dist}.)

The key point is that the blow-down map $\V' \to \V$ is a Poisson morphism; see \autoref{rem:bd-uni}.  Let $\Y = \Sec[k]{\X}\setminus\Sec[k-1]{\X}$ and $\Y' = \pi^{-1}(\Y)$.  Then the blow-down identifies $\Y'\setminus 0 $ with the cone over $\Y$ in $\V$, which is a $2k$-dimensional symplectic leaf of the Poisson structure.  Hence  $\Y'\setminus 0$ is a symplectic leaf in $\V'$.  In particular, it is a Poisson subvariety, and so $\Y \subset \Prj{d-1}$ is a Poisson subscheme to which $\strc{\Prj{\mathit{d}-1}}(-1)$ restricts as a Poisson module. It follows that the residue is tangent to $\Y$.  
By \autoref{rem:adapt-lb}, the residue is actually non-vanishing on $\Y$.

To conclude that the residue is gives a non-vanishing multiderivation on all of $\Sec[k]{\X}$, we note that the sheaf $\der[2k+1]{{\Sec[k]{\X}}}$ is reflexive.  We know from \autoref{thm:secants} that $\Sec[k]{\X}$ is normal and $\Sec[k-1]{\X} = \Sec[k]{\X}\setminus \Y$ has codimension two, so it follows from~\cite{Hartshorne1980} that the residue on $\Y$ extends to a non-vanishing multiderivation on all of $\Sec[k]{\X}$.
\end{proof}

\begin{remark}\label{rem:bd-uni}
The fact that the blow-down is a Poisson morphism in this case is not entirely obvious.  As will be explained in a forthcoming paper~\cite{Pym2012} by the second author, we need to know that the Poisson structure on $\V$ is \defn{unimodular}.  For the examples in question, the unimodularity was proven in~\cite{Ortenzi2011}.\qed
\end{remark}

\begin{remark}
This proposition shows that in every dimension, there are examples of Poisson manifolds whose modular residues are all nonzero.  We therefore argue that these residues provide a possible explanation for the dimensions appearing in Bondal's conjecture.  Since these residues are non-vanishing on the secant varieties, it may be reasonable to expect the degeneracy loci of Poisson structures to be highly singular Calabi-Yau varieties.\qed
\end{remark}

We close with a conjecture regarding the degeneracy loci of the Fe{\u\i}gin-Odesski{\u\i} Poisson structures, which we hope to address issue in future work.

\begin{conjecture}\label{con:reduced}
The degeneracy loci of the Fe{\u\i}gin-Odesski{\u\i} Poisson structures on $\Prj{d-1}$  are reduced.
\end{conjecture}

\begin{proof}[Evidence]
We focus on the case of Poisson structures on $\Prj{}(\cohlgy[1]{\X,\sL^\vee})$ for $\sL$ a line bundle of degree $d$.

According to \autoref{prop:res-nonvan}, the modular residue on every degeneracy locus $\D[2k]{\ps}$ is non-trivial.   This is a sort of genericity condition on the one-jet of $\ps^{k+1}$, and hence it may be reasonable to expect that $\ps^{k+1}$ vanishes transversally.  For $d=2n+1$ an odd number, the degeneracy divisor $\D[2n-2]{\ps}$ is always reduced, as we saw in \autoref{thm:sec-sing}.  In general, though, the non-triviality of the residue is not sufficient to conclude that $\D[2k]{\ps}$ is reduced. 

We also have explicit knowledge in low dimensions: for $d=3$, the zero locus is the smooth cubic curve $\X \subset \Prj{2}$, which is the reduced anti-canonical divisor.

For $d=4$, the formulae in \cite{Odesskii2002,Ortenzi2011} show that the two-dimensional symplectic leaves form a pencil of quadrics in $\Prj{3}$ intersecting in a smooth elliptic curve.  The zero scheme of the Poisson structure is reduced and consists of the elliptic curve, together with four isolated points~\cite{Polishchuk1997}.

For $d=5$, we checked using \texttt{Macaulay2}~\cite{M2} and the formulae in~\cite{Odesskii2002,Ortenzi2011} that the degeneracy loci are reduced.
\end{proof}

\bibliographystyle{hyperamsplain}
\bibliography{bondal4d}

\end{document}